\documentclass[12pt]{amsart}
\usepackage{amsmath,amsthm,amsopn,amsfonts,amssymb,color}
\usepackage[colorlinks=true,linkcolor=blue,urlcolor=blue]{hyperref} 
\usepackage{cancel}
 
\usepackage{graphicx}
\usepackage{float}
\usepackage{amsmath}
\usepackage{fancyhdr}
\usepackage{epstopdf}
\usepackage{amsfonts}

\usepackage{booktabs}
\usepackage[USenglish]{babel}
\usepackage{cancel}
\usepackage{times}
\usepackage{mathrsfs}
\usepackage{multirow,multicol}
\usepackage{amsfonts}
\usepackage{amsthm}
\usepackage{amsmath}
\usepackage{amsmath,amssymb}
\newtheorem{teo}{Theorem}[section]
\newtheorem{prop}[teo]{Proposition}

\newtheorem{rem}[teo]{Remark}

\newtheorem{lemma}[teo]{Lemma}

\newcommand{\infi}{\mathcal{1}}

\def\ioo#1{\int_{\{| w_n|>{#1}\}}}

\def\elle#1{L^{#1}(\Omega)}
\def\emme#1{M^{#1}(\Omega)}
\def\lor#1#2{L^{#1,#2}(\Omega)}
\def\lorr#1#2{\mathbb{L}^{#1,#2}(\Omega)}
\def\diver{{\rm div}}
\def\elle#1{L^{#1}(\Omega)}
\def\w{H_0^{1}(\Omega)}
\def\io{\int_{\Omega}}
\def\norma#1#2{\|#1\|_{\lower 4pt \hbox{$\scriptstyle #2$}}}
\def\un{u_n}

\def\wn{w_n}

\def\fn{f_n}

\def\Gk{G_k}
\def\R{I \!\!R}

\def\elle#1{L^{#1}(\Omega)}
\def\emme#1{M^{#1}(\Omega)}
\def\w{W_0^{1,2}(\Omega)}

\def\w1{W_0^{1,1}(\Omega)}

\def\w{H_0^{1}(\Omega)}

\def\be{\begin{equation}}
\def\ee{\end{equation}}

 \usepackage{color}
 \numberwithin{equation}{section}

\author[S. Buccheri]{Stefano Buccheri}
\address{Dipartimento di Matematica, "Sapienza" Universit\`a di Roma, Piazzale Aldo Moro, 00100 Rome, Italy.}
\email{buccheri@mat.uniroma1.it}

\title[]{ Gradient estimates for nonlinear elliptic equations with first order terms.}
\keywords{Rearrangement of the gradient, non coercive problems, convection term, drift term.} \subjclass[2010]{35J25, 35J60}

\date{\today}

\begin{document}

\thispagestyle{empty}

\maketitle
\begin{abstract}
We study existence and Lorentz regularity of distributional solutions to elliptic equations with either a convection or a drift first order term. The presence of such a term makes the problem not coercive. The main tools are pointwise estimates of the rearrangements of both the solution and its gradient. 
\end{abstract}
\tableofcontents

\section{Introduction and model problems}
This paper is concerned with the study of existence and Lorentz regularity of distributional solutions to a class of non coercive nonliear elliptic partial differential equations with Dirichlet boundary conditions. The non coercivity is given by the presence of first order terms. To avoid technicalities, in the introduction we present the linear version of such equations, while the general case is treated in Section \ref{mainresults}.\\ Let us consider at first the following model problem
\begin{equation} \label{convection}
\begin{cases}
-\mbox{div}\left(A(x)\nabla u\right)=-\mbox{div}\left(u\,E(x)\right)+f(x) & \mbox{in $\Omega$,} \\
\hfill u = 0 \hfill & \mbox{on $\partial\Omega$,}
\end{cases}
\end{equation}
where $\Omega$ is a bounded open set of $\mathbb{R}^N$, with $N>2$, $A(x)$ is a matrix with measurable coefficients that satisfies for $\alpha, \beta>0$
\begin{equation}\label{00}
\alpha|\xi|^2\le A(x)\xi\cdot\xi, \quad |A(x)| \leq\beta,  \quad  \mbox{a.e.}\; x\in\Omega,\quad \forall \;\xi\in\R^N,
\end{equation}
the vector field $E(x)$ belongs either to the Lebesgue or the Marcinkiewicz space of order $N$ and the function $f(x)$ belongs to a suitable Lorentz space to be precised (see Section \ref{mainresults} for the definition of these spaces). In the literature the lower order term in divergence form of \eqref{convection} is often called \emph{convection term}.\\ If $f\in\elle{(2^*)'}$ we can consider the \emph{weak formulation} of problem \eqref{convection}, namely
\begin{equation}\label{13:37}
u\in W^{1,2}_0(\Omega) \ : \ \io A(x)\nabla u\nabla\phi=\io u E(x)\nabla\phi+ \io f(x)\phi \ \ \ \forall \ \phi\in W^{1,2}_0(\Omega).
\end{equation}
The assumption 
\begin{equation}  \label{!econ}
 E(x)\in\left(\elle{N}\right)^N
\end{equation}
assures that the convection term of \eqref{13:37} is well defined since
\begin{equation}\label{27-11}
v E(x)\nabla \phi \in\elle{1} \ \ \ \forall \ v,\phi\in W^{1,2}_0(\Omega).
\end{equation}
Notice that \eqref{!econ} is not the most general condition in order to have \eqref{27-11}. Indeed if we assume
\begin{equation}\label{27-11bis}
E(x)\in\left(\emme{N}\right)^N,
\end{equation}
it follows that \eqref{27-11} continues to hold, thanks to the sharp Sobolev Embedding in Lorentz spaces, namely
\[
W^{1,q}_0(\Omega)\subset \lor{q^*}{q} \ \ \ \mbox{with} \ \ \ 1\le q<\infty.
\]
If $f\in\elle 1$, problem \eqref{convection} has to be meant trough the following \emph{distributional formulation}
\begin{equation}\label{!16:57}
u\in W^{1,1,}_0(\Omega) \ : \ \io A(x)\nabla u\nabla\phi=\io u E(x)\nabla\phi+ \io f(x)\phi \ \ \ \forall \ \phi\in C^1_0(\Omega).
\end{equation}
Notice again that, assuming either \eqref{!econ} or \eqref{27-11bis}, we have that $v E(x)\in\elle1$ for any $v\in W^{1,1,}_0(\Omega)$.\\

The main feature of \eqref{convection} is the non coercivity of the convection term, as it can be seen with the following heuristic argument. Assuming for simplicity \eqref{!econ}, $f\in\elle{(2^*)'}$ and letting $u\in W^{1,2}_0(\Omega)$ be the solution of \eqref{13:37}, we obtain that 
\[
\alpha\|u\|^2_{W^{1,2}_0(\Omega)}\le \frac{\|E\|_{\elle{N}}}{\mathcal{S}}\|u\|^2_{W^{1,2}_0(\Omega)}+\|f\|_{\elle{\left(2^{*}\right)'}}\|u\|_{\elle{2^*}},
\] 
where $\mathcal{S}$ denotes the Sobolev constant. Thus if the value of $\|E\|_{\elle N}$ is large, the first term in the right hand side above cannot be absorbed in the left hand one.\\

The classical approach in dealing with \eqref{convection} (see for instance \cite{S66}, \cite{LU} and \cite{TRU0}) is to assume a smallness condition on the $\elle N$-norm, 
 \begin{equation}\label{smal0}
\|E\|_{\elle N}<\mathcal{S}\alpha,
\end{equation}
or a sign condition on the distributional divergence of $E(x)$,
\begin{equation}\label{zerodiv0}
\io E(x)\nabla\phi\ge0 \ \ \ \forall \  \phi\in C^1_0(\Omega),
\end{equation}
so that the problem becomes coercive. Alternatively, to restore the lack of coercivity, one can add an \emph{absorption term} in the left hand side of \eqref{convection} (see for instance \cite{S66} or the more recent \cite{DD}).\\

One naturally wonders if assumptions \eqref{smal0} or \eqref{zerodiv0} are necessary or rather it is possible to achieve a priori estimates for the solution of \eqref{convection} even if the associated operator is not coercive. The answer is given in \cite{DP} and \cite{B09} where it has been proven the following result.
\begin{teo}[\cite{DP}, \cite{B09}]\label{!teo0}
Let us assume \eqref{00}, $E\in\left(\elle N\right)^N$ and that $f\in\elle{m}$ with $1< m<\frac{N}{2}$. Then\\
(i) if $(2^*)'\le m<\frac{N}{2}$ there exists $u\in\elle{m^{**}}\cap W^{1,2}_0(\Omega)$ solution of \eqref{13:37};\\
(ii) if $1<m< (2^*)'$ there exists $u\in W^{1,m^*}_0(\Omega)$ solution of \eqref{!16:57}.

\end{teo}
Thus, not only problem \eqref{convection} is solvable in $W^{1,2}_0(\Omega)$ for any vector field $E$ satisfying \eqref{!econ} (no matter the size of its norm), but also the same sharp regularity result of the case $E\equiv0$ (see \cite{bg2}) is recovered, even for distributional solutions with data outside $W^{-1,2}(\Omega)$. Let us also mention \cite{DPbis}, for similar results with more restrictive assumptions on the summability of $E(x)$.\\ We stress that, even if Theorem \ref{!teo0} is stated for a linear problem, in \cite{DP}, \cite{DPbis} a more general non linear versions of \eqref{convection} is treated. Moreover \cite{DP} and \cite{DPbis} consider an equation with both convection and drift (see \eqref{!drift} below) first order terms, assuming a smallness condition on at least one of them. We do not treat these two lower order terms together and the reason is explained at the end of this section.\\

Let us briefly describe the approaches used in \cite{B09} and \cite{DP} in order to deal with problem \eqref{convection}. The strategy of the first paper hings on an a priori estimate on the measure of the super level sets of $u$. Such estimate \emph{bypasses} in some sense the non coercivity of the problem and allows the author to recover the integral estimates for $u$ and $|\nabla u|$.\\
On the other hand, in \cite{DP} (see also \cite{DPbis}) the authors approach problem \eqref{convection} by symmetrization technique (see \cite{T}): the main idea is to deduce a differential inequality for the decreasing rearrangement of $u$ (see Sections \ref{mainresults} and \ref{unoo} for a brief introduction on this issue) and compere it with the rearrangement of the solution of a suitable symmetrized problem. Since the solution of the symmetrized problem is explicit one recovers the a priori estimate for $u$ and, in turn, the energy estimate for the gradient. Let us stress that such a symmetrization approach does not provide any information about the regularity of $|\nabla u|$.\\

Our main contribution for problem \eqref{convection} (and its nonlinear counterpart) is to complete the relation between the regularity of $f$ and $u, \, |\nabla u|$ in the framework of Lorentz spaces under optimal conditions on the summability of $E(x)$. More in detail we provide the following result (see Theorems \ref{teohighp} and \ref{teolowp} in Section \ref{mainresults} for the general case).
\begin{teo}\label{!!teo1}
Assume \eqref{00},  $f\in\lor{m}{q}$, with $1 <m <\frac{N}{2}$, $0<q\le\infty$, $E\in\left(\emme N\right)^N$ and moreover that there exist
\begin{equation}\label{EV}
\mathcal{F}\in\left(\elle{\infty}\right)^N, \ \ \mathcal{E}\in\left(\emme{N}\right)^N \ \ \ \mbox{with} \ \ \ \|\mathcal{E}\|_{\emme N}<\alpha\frac{N-2m}{m},
\end{equation}
such that $E=\mathcal{F}+\mathcal{E}$.
Hence there exists $u$  solution of \eqref{!16:57}. Moreover 
\begin{itemize}
\item if $1 <m <(2^*)'$, then $u\in \lor{\frac{Nm}{N-2m}}{q}$ and $|\nabla u|\in \lor{\frac{Nm}{N-m}}{q}$;
\item if $(2^*)'< m<\frac{N}{2}$, then $u\in W^{1,2}_0(\Omega)\cap\lor{\frac{Nm}{N-2m}}{q}$.
\end{itemize}
\end{teo}
Let us briefly comment this result. The more interesting (and difficult) part of its proof is the first one ($1 <m <(2^*)'$), where the regularity of the gradient increases with the regularity of the datum. To prove it we need pointwise estimates not only for $\overline{u}$, the decreasing rearrangement of $u$, but also for $\overline{|\nabla u|}$, the decreasing rearrangement of $\nabla u$. Let us stress  again that, while estimates of $\overline{u}$ are already known in the literature for problems similar to \eqref{convection} (see for instance \cite{BM} for the case $(2^*)'< m<\frac{N}{2}$ and also \cite{DPbis} \cite{DP}), the estimate for $\overline{|\nabla u|}$ is new.\\
Moreover assumption \eqref{EV} is optimal in the sense that, if \[\|\mathcal{E}\|_{\emme N}\ge \alpha\omega_N^{\frac{1}{N}}\frac{N-2m}{m},\] the standard relation between the regularity of $f$ and $u$ is lost and the regularity of the solution depends on the value of the Marcinkiewicz norm of $\mathcal{E}$ (see \cite{BO} and Remark \ref{nirvana}). Let us also notice that the \eqref{EV} is more general then \eqref{!econ}.\\
In Section \ref{mainresults} we generalize Theorem \ref{!!teo1} considering a more general non linear operator. In this nonlinear setting we also deal with solutions in $W^{1,1}_0(\Omega)$. This represent an additional difficulty due to the lack of compactness of bounded sequences in such a space.\\

 For the same result in the case $E\equiv0$ we refer to \cite{B9} for $f$ in Marcinkiewicz spaces (see also \cite{KL}) and \cite{AFT} for data in Lorentz spaces. We have also to mention that unfortunately our approach does not cover the case $m=(2^*)'$. This borderline case has been solved by \cite{Min}, if $E\equiv 0$, using non standard (nonlinear) potential arguments.
\\

An example of the second type of problems that we consider is 
\begin{equation}\label{!drift}
\begin{cases}
-\diver (A(x)\nabla w) =  E(x) \nabla w + f(x) & \mbox{in $\Omega$,} \\
\hfill w = 0 \hfill & \mbox{on $\partial\Omega$},
\end{cases}
\end{equation}
with $A(x)$ satisfying \eqref{00}, $E(x)$ as in \eqref{!econ} or \eqref{27-11bis} and $f$ that belongs to a Lorentz space. The first order term in the equation above is also called {\it drift} term. \\

In this linear setting \eqref{!drift} is (at least formally) the dual problem of \eqref{convection} and one can use a duality approach to recover existence and regularity results (see \cite{DD}, \cite{B18bis}, \cite{BBC}). Anyway here we treat problem \eqref{!drift} independently from \eqref{convection}, following the same spirit and aims of the previous case.\\
Similarly to the convection one, also the drift term makes the operator of \eqref{!drift} not coercive, unless an additional smallness assumption on the $\elle N$ norm of $E(x)$ is assumed. Once again it is proved that such assumption is unnecessary for the existence of a weak solution, see \cite{BM}, \cite{DPbis} and \cite{DP}.
While in the last three papers problem \eqref{!drift} is studied with symmetrization techniques, in \cite{BM} the authors obtain energy estimates for \eqref{!drift} by means of a \emph{slice method} that is based on continuity properties of some modified distribution function of $w$ (see \cite{BMMP} and the more recent \cite{Bucc} for related results).\\

Here we adapt the techniques developed for problem \eqref{convection} to recover Lorentz regularity results also for problem \eqref{!drift}. Being mainly interested in solution outside the energy space, let us introduce the distributional formulation of \eqref{!drift}.
\begin{equation}\label{!!17:38}
w\in W^{1,r}_0(\Omega) \ : \ \io A(x)\nabla w\nabla\phi=\io E(x)\nabla w \phi+ \io f(x)\phi \ \ \ \forall \ \phi\in C^1_0(\Omega),
\end{equation}
with $r>\frac{N}{N-1}$. Notice that we have to impose that $w\in W^{1,r}_0(\Omega)$, with $r>\frac{N}{N-1}$, so that the lower order term of \eqref{!!17:38} is well defined.\\ Also in this case the key point is to obtain pointwise estimate for $\overline{w}$ and $\overline{|\nabla w|}$, the decreasing rearrangements of $w$ and $|\nabla w|$. Let us state the existence and regularity result for problem \eqref{!!17:38}
\begin{teo}\label{!!teo2}
Assume \eqref{00},  $f\in\lor{m}{q}$, with $1 <m <\frac{N}{2}$, $0<q\le\infty$,$E\in\left(\emme N\right)^N$ and moreover that there exist
\begin{equation*}
\mathcal{F}\in\left(\elle{\infty}\right)^N, \ \ \mathcal{E}\in\left(\emme{N}\right)^N \ \ \ \mbox{with} \ \ \ \|\mathcal{E}\|_{\emme N}<\alpha N\frac{m-1}{m},
\end{equation*}
such that $E=\mathcal{F}+\mathcal{E}$. Hence there exists $w$ distributional solution of \eqref{!!17:38}. Moreover 
\begin{itemize}
\item if $1 <m <(2^*)'$, then $w\in \lor{\frac{Nm}{N-2m}}{q}$ and $|\nabla w|\in \lor{\frac{Nm}{N-m}}{q}$;
\item if $(2^*)'< m<\frac{N}{2}$, then $w\in W^{1,2}_0(\Omega)\cap\lor{\frac{Nm}{N-2m}}{q}$.
\end{itemize}
\end{teo}

We refer to the next Section \ref{mainresults} for the nonlinear version of Theorems \eqref{convection} and \eqref{!drift}.\\

After studying problems \eqref{convection} and \eqref{!drift} separately, one naturally wonders why do not consider the convection and the drift terms at once. This is what is actually done in \cite{S66}, \cite{TRU0} and \cite{DP} but still imposing some additional constraints, as smallness assumptions on the $L^N$ norm of at least one of the vector field or divergence free assumptions, as \eqref{zerodiv0}. One may wonder if, also in this case, these are just technical assumptions, or rather the presence of the two first order term represents a genuine obstruction to the solvability of the following problem
\begin{equation}\label{intro}
\begin{cases}
-\diver (A(x)\nabla u) =-\mbox{div}\left(u E(x)\right)+ B(x)\nabla u + f(x) & \mbox{in $\Omega$,} \\
\hfill u = 0 \hfill & \mbox{on $\partial\Omega$}.
\end{cases}
\end{equation} 
Let us observe that, in the special case $E, B \in \left(C^{1}(\Omega)\right)^N $ and $E\equiv B$, problem \eqref{intro} becomes 
\begin{equation*}
\begin{cases}
-\diver (A(x)\nabla u) =g(x)u + f(x) & \mbox{in $\Omega$,} \\
\hfill u = 0 \hfill & \mbox{on $\partial\Omega$}.
\end{cases}
\end{equation*}
with $g(x)=-\mbox{div}(E(x))$, that of course is not solvable for a general $E(x)\in\left(C^1(\Omega)\right)^N$. Thus the presence of the two lower order terms involve some spectral issues and we do not treat it.\\

\section{Main results}\label{mainresults}
In order to state our main results in their full generality, we need to introduce some basic definitions and properties about rearrangements and Lorentz spaces.\\ For any measurable function $v:\Omega\to\mathbb{R}$, we define the \emph{distribution function} of $v$ as
\[
A(t):=|\{x\in\Omega \ : \ |v(x)|>t \ \}| \ \ \ \mbox{for} \ t\ge0,
\]
and the \emph{decreasing rearrangement} of $v$ as
\[
\overline{v}(s):= \inf \{t\ge0\ : \ A(t)<s\} \ \ \ \mbox{for} \ s\in[0,|\Omega|].
\]
By construction it follows that
\begin{equation}\label{equimis}
|\{x\in\Omega \ : \ |v(x)|>t\}|=|\{s\in\mathbb{R} \ : \ \overline{v}(s)>t\}|,
\end{equation}
namely the function and its decreasing rearrangement are equimeasurable. We define also the \emph{maximal function} associated to $\overline{v}$ as
\[
\tilde{v}(s)=\frac{1}{s}\int_0^s\overline{v}(t)dt.
\]
Notice that, since $\overline{v}(s)$ is non increasing, it follows that $\overline{v}(s)\le \tilde{v}(s)$ for any $s\in[0,|\Omega|]$.\\
By definition $A(t)$ is right continuous and non increasing, while $\overline{v}(s)$ is left continuous and non increasing. Thus both functions are  almost everywhere differentiable in $(0,|\Omega|)$. For a more detailed treatment of $A(t)$ and $\overline{v}(s)$ we refer to \cite{Okl} and \cite{Gra}.\\

Let us give now the definition of Lorentz spaces. For $1\le m<\infty$ and $0<q\le\infty$ we say that a measurable functions $f:\Omega\to \mathbb{R}$ belongs to the Lorentz space $L^{m,q}(\Omega)$ if the quantity
\[
\| f \|_{L^{m,q}(\Omega)}=\begin{cases} \left(\int_0^{\infty}t^{\frac{q}{m}}\overline{f}(t)^{q}\frac{dt}{t}\right)^{\frac{1}{q}} &\mbox{if }   q<\infi\\ \sup_{t\in(0,\infty)} t^{\frac{1}{m}}\overline{f}(t)  &\mbox{if } q=\infi,\end{cases}
\]
is finite. We recall that $\lor{m}{m}=\elle m$ and that
\[
\lor{m}{q}\subset\lor{m}{r}  \ \ \ \mbox{for any} \ \ \ 0<q<r\le \infty.
\]
The space $\lor{m}{\infty}$, with $1\le m <\infty$ is called Marcinkiewicz space of order $m$ and we denote it by $\emme m$.\\
If we replace $\overline{f}$ with $\tilde{f}$, we define another space $L^{(m,q)}(\Omega)$ given by all the measurable function $f:\Omega\to \mathbb{R}$ such that the quantity
\[
\lceil f\rceil_{L^{(m,q)}(\Omega)}=\begin{cases} \left(\int_0^{\infty}t^{\frac{q}{m}}\tilde{f}(t)^{q}\frac{dt}{t}\right)^{\frac{1}{q}} &\mbox{if }   q<\infi\\ \sup_{t\in(0,\infty)} t^{\frac{1}{m}}\tilde{f}(t)  &\mbox{if } q=\infi\end{cases}
\]
is finite. Since
\begin{equation}\label{equivalent}
\| f \|_{L^{m,q}(\Omega)}\le \lceil f\rceil_{L^{(m,q)}(\Omega)}\le m' \| f \|_{L^{m,q}(\Omega)},
\end{equation}
it results that $\|\cdot\|_{L^{m,q}(\Omega)}$ and $\lceil \cdot\rceil_{L^{(m,q)}(\Omega)}$ are equivalent if $m>1$ and $\lor{m}{q}\equiv L^{(m,q)}(\Omega)$.  Anyway in the borderline case $m=1$ the space $L^{(1,q)}(\Omega)$ is rather unsatisfactory since, for $q<\infty$, it contains only the zero function. This is because by definition $\tilde{f}(s)\approx\frac{1}{s}$ for $s>|\Omega|$. Hence, following \cite{CB}, we define $\mathbb{L}^{1,q}(\Omega)$ as the set of measurable function $f$ such that 
\[
\|f\|_{\mathbb{L}^{1,q}(\Omega)}=\begin{cases} \left(\int_0^{|\Omega|}t^q\tilde{f}(t)^{q}\frac{dt}{t}\right)^{\frac{1}{q}} &\mbox{if }   q<\infi\\ \sup_{t\in(0,|\Omega|)} t\tilde{f}(t)  &\mbox{if } q=\infi,\end{cases}
\]
is finite. Notice that in \cite{CB} is proved that $f$ belongs to $\lorr11$ if and only if
\[
\int |f|\log(1+|f|)<\infty.
\]
Hence $\lorr{1}{1}\equiv L\log L(\Omega)$, while the space $\lorr{1}{q}$ with $1<q<\infty$ is a \emph{diagonal intermediate} space between $L\log L(\Omega)$ and $\elle1$ (see \cite{CB}).\\
 
Let us present now our first problem in its general form. Given $1<p<N$, consider
\begin{equation} \label{pdrift}
\begin{cases}
-\mbox{div}\left(a(x,\nabla u)\right)=-\mbox{div}\big(u|u|^{p-2} E(x)\big)+f(x) & \mbox{in $\Omega$,} \\
\hfill u = 0 \hfill & \mbox{on $\partial\Omega$,}
\end{cases}
\end{equation}
where the Carath\'eodory function $a:\Omega\times\mathbb{R}^{N}\rightarrow\mathbb{R^{N}}$
 satisfies for $0<\alpha,\beta$ 
\begin{equation} \label{ll5}
\begin{split}
\alpha&|\xi|^{p}\le \  a(x,\xi)\xi  ,\\
| a&(x,\xi) | \le \ \beta |\xi|^{p-1}, \\
[a&(x,\xi)-a(x,\xi^{*})][\xi-\xi^{*}]> \ 0, \ \ \ \mbox{if} \ \ \ \xi\neq\xi^*,
\end{split}
\end{equation}
the datum $f$ belongs to $\lor mq$ with $1\le m<\frac{N}{p}$, $0<q\le\infty$ and the vector field $E:\Omega\rightarrow\mathbb{R^{N}}$ is such that
\begin{equation}\label{marp}
E=\mathcal{F}+\mathcal{E} \ \ \mbox{ with } \ \ \mathcal{F}\in\left(\elle{\infty}\right)^N \ \ \mbox{ and } \ \ \overline{\mathcal{E}}(s)\le \frac{B}{s^{\frac{p-1}{N}}} \ \ \ \mbox{with}\ \ \  B<\alpha^{\frac1{p-1}}\omega_N^{\frac{1}{N}}\frac{N-pm}{(p-1)m}.
\end{equation}
Assumption \eqref{marp}, up to the addition of a whichever bounded vector filed, prescribes a threshold on the $\emme{N}$-norm of $E$ (see also \cite{BM} and \cite{B17}). As already said in the introduction, this smallness condition is sharp and cannot be weakened (see \cite{BO} and Remark \ref{nirvana}). To clarify the different notation between \eqref{EV} and \eqref{marp}, let us notice that $\|\mathcal{E}\|_{\emme N}\omega_N^{\frac1N}\le B$. \\
Let us introduce the distributional formulation of Problem \eqref{pdrift}.
\begin{multline}\label{distrpp}
\ \ \ \ \ \ \ \ \ \ \ u\in W^{1,1}_0(\Omega) \ :
\begin{split}
|\nabla u|^{p-1}\in\elle1, \ \ \ |u&|^{p-1}|E(x)|\in\elle1 \ \ \ \mbox{and}\\
\io a(x,\nabla u)\nabla\phi=\io |u|^{p-2}u&\, E(x)\nabla\phi+\io f(x)\phi \ \ \ \forall \ \phi\in C^{1}_0(\Omega).
\end{split} 
\end{multline}
Let us state the first result of this section.
\begin{teo}\label{teohighp}
Let us assume $f\in L^{m,q}(\Omega)$ and that conditions \eqref{ll5} and \eqref{marp} hold true.\\ (i) If $\max\{1,\frac{N}{N(p-1)+1}\}<m<\left(p^*\right)'$ and $0<q\le\infi$, then there exists $u$ solution of \eqref{distrpp} such that
\[
|u|\in L^{\frac{(p-1)Nm}{N-mp},(p-1)q}(\Omega) \ \ \ \mbox{and} \ \ \  |\nabla u|\in L^{\frac{(p-1)Nm}{N-m},(p-1)q}(\Omega).
\]
(ii) If $\left(p^*\right)'<m<\frac{N}{p}$ and $0<q\le\infi$
\[
|u|\in L^{\frac{(p-1)Nm}{N-mp},(p-1)q}(\Omega)\cap W^{1,p}_0(\Omega).
\]
\end{teo}
As already said in the introduction the main novelty of this Theorem is the first part (see \cite{BM} for similar result in the case $\left(p^*\right)'<m<\frac{N}{p}$ and $0<q\le\infi$) and the core of the proof relies on the estimate on the decreasing rearrangement of the gradient provided in Lemma \ref{regdup} of Section \ref{proofof}.\\ 
In order to present the second result of this section, let us recall that, if $p$ and $m$ are close to $1$, some subtleties arise (see \cite{BBBB} for the case $E\equiv0$). Roughly speaking this is because the gradient of the expected solution might not be an integrable function. Indeed, if $1<p<2-\frac1N$ and $1<m<\frac{N}{N(p-1)+1}$, the notion of distributional solution is not any more adequate and entropy solutions have to be introduced (see for instance \cite{B17}). We do not treat this case and instead focus on the bordeline values $m=\max\big\{1,\frac{N}{N(p-1)+1}\big\}$. In the cases $f\in\elle1$ or $f\in \lorr{1}{q}$, we assume $m=1$ in \eqref{marp}.
\begin{teo}\label{teolowp}
Let us assume $m=\max\{1,\frac{N}{N(p-1)+1}\}$ and that conditions \eqref{ll5} and \eqref{marp} hold true. Hence there exists $u$ solution of \eqref{distrpp}. Moreover\\ (i) if $p>2-\frac{1}{N}$ and $f\in\elle1$, then 
\[
|u|\in L^{\frac{(p-1)N}{N-p},\infi}(\Omega) \ \ \ \mbox{and} \ \ \  |\nabla u|\in L^{\frac{(p-1)N}{N-1},\infi}(\Omega);
\]
(ii) if $p>2-\frac{1}{N}$ and $f\in\lorr{1}{q}$ with $0< q\le\infty$, then
\[
|u|\in L^{\frac{(p-1)N}{N-p}, (p-1)q}(\Omega) \ \ \ \mbox{and} \ \ \  |\nabla u|\in L^{\frac{(p-1)N}{N-1},(p-1)q}(\Omega);
\]
(iii) if $p=2-\frac{1}{N}$ and $f\in\mathbb{L}^{1,q}(\Omega)$ with $0< q\le\frac{1}{p-1}=\frac{N}{N-1}$, then 
\[
|u|\in L^{\frac{N}{N-1}, \frac{N-1}{N}q}(\Omega) \ \ \ \mbox{and} \ \ \  |\nabla u|\in L^{1,(p-1)q}(\Omega);
\]
(iv) if $p<2-\frac{1}{N}$ and $f\in\lor{m}{q}$ with $m=\frac{N}{N(p-1)+1}$ $0< q\le\frac{1}{p-1}$, then 
\[
|u|\in L^{\frac{N}{N-1},(p-1)q}(\Omega) \ \ \ \mbox{and} \ \ \  |\nabla u|\in L^{1,(p-1)q}(\Omega).
\]
\end{teo}
The main observation on Theorems \ref{teohighp} and \ref{teolowp} is that, also in this nonlinear Lorentz setting, we recover the same results of the case $E\equiv0$ (see \cite{AFT}, \cite{BG}, \cite{Min} and reference therein). Let us further comment Theorem \ref{teolowp}.
In $(i)$ and $(ii)$ the summability of the data assures that $|\nabla u|$ belongs to a Lebesgue space smaller (more regular) than $\elle 1$. On the contrary, in \emph{(iii)} and \emph{(iv)}, the gradient belongs to Lorentz spaces with first exponent equal to $1$. Such spaces are contained at most in $\elle1$ and this makes more difficult the proof since $\elle1$ is not reflexive (we refer to \cite{BG} for corresponding results restricted to the Lebesgue framework with $E\equiv0$).\\

Finally let us focus on nonlinear drift term. Let us consider, for $p>1$,
\begin{equation} \label{!!modelcase}
\begin{cases}
-\mbox{div}\left(a(x,\nabla w)\right)=E(x)|\nabla w|^{p-2}\nabla w +f(x) & \mbox{in $\Omega$,} \\
\hfill u = 0 \hfill & \mbox{on $\partial\Omega$,}
\end{cases}
\end{equation}
where the Carath\'eodory function \emph{a} : $\Omega\times\mathbb{R}^{N}\rightarrow\mathbb{R^{N}}$ satisfies \eqref{ll5}, the datum $f$ belongs to $\lor mq$ with $1\le m<\frac{N}{p}$, $0<q\le\infty$ and the vector field $E:\Omega\rightarrow\mathbb{R^{N}}$ is such that there
\begin{equation}\label{!!marbis}
E=\mathcal{F}+\mathcal{E} \ \ \mbox{ with } \ \ \mathcal{F}\in\left(\elle{\infty}\right)^N \ \ \mbox{ and } \ \ \overline{\mathcal{E}}(s)\le \frac{B}{s^{\frac{1}{N}}} \ \ \ \mbox{with}\ \ \  B<\alpha \omega_N^{\frac{1}{N}} N \frac{m-1}{m}.
\end{equation}
Let us recall again that $\|\mathcal{E}\|_{\emme N}\omega_N^{\frac1N}\le B$. It is immediate to note that this assumption becomes more and more restrictive as $m$ approaches $1$. This is not just a technical inconvenient and prevent us to treat the case $f$ in $\elle1$ or $\lorr{1}{q}$ with $1<q<\infty$. Indeed, for such type of data and assuming \eqref{!!marbis}, the expected regularity of the gradient is too low to have the drift term of \eqref{!!modelcase} well defined (we refer the interested reader to \cite{BMMP}).
We consider the following weak formulation of problem \eqref{!!modelcase}.
\begin{multline}\label{!!pdriftw}
\ \ \ \ \ \ \ \ \ \ \ u\in W^{1,1}_0(\Omega) \ :
\begin{split}
|\nabla u|^{p-1}\in\elle1, \ \ \ |E(x)||\nabla u&|^{p-1}\in\elle1 \ \ \ \mbox{and}\\
\io a(x,\nabla u)\nabla\phi=\io E(x)|\nabla &u|^{p-2}\nabla u  \phi+\io f(x)\phi \ \ \ \forall \ \phi\in C^{1}_0(\Omega).
\end{split}
\end{multline}
Let us state the existence and regularity result for problem \eqref{!!pdriftw}.
\begin{teo}\label{!!teohighp}
Let us assume $f\in L^{m,q}(\Omega)$ and that conditions \eqref{ll5} and \eqref{!!marbis} hold true.\\
(i) If $\max\{1,\frac{N}{N(p-1)+1}\}<m<\left(p^*\right)'$ and $0<q\le\infi$, then there exist $u$ solution of \eqref{!!pdriftw} such that
\[
|u|\in L^{\frac{(p-1)Nm}{N-pm},(p-1)q}(\Omega) \ \ \ \mbox{and} \ \ \  |\nabla u|\in L^{\frac{(p-1)Nm}{N-m},(p-1)q}(\Omega).
\]
(ii) If $\left(p^*\right)'<m<\frac{N}{p}$ and $0<q\le\infi$
\[
|w|\in L^{\frac{(p-1)Nm}{N-mp},(p-1)q}(\Omega)\cap W^{1,p}_0(\Omega).
\]
\end{teo}
Schematically the strategy of the proof of Theorems \ref{teohighp}, \ref{teolowp} and \ref{!!teohighp} consists of the following steps:
\begin{itemize}
\item finding suitable sequence of approximating solutions $\{\un\}$ and $\{\wn\}$ for problem \eqref{distrpp} and \eqref{!!pdriftw} respectively;
\item a priori estimates for the sequences $\{\un\}$ and $\{\wn\}$ in the required Lorentz spaces;
\item existence of a converging subsequences to weak limits $u$ and $v$;
\item passage to the limit as $n\to\infty$ to prove that $u$ and $v$ are indeed solutions of the initial problems.
\end{itemize}
The first step is obtained \emph{truncating} problems \eqref{distrpp} and \eqref{!!pdriftw}. Indeed thanks to \cite{ll}, for any $n\in\mathbb{N}$ we infer the existence of $\un\in W^{1,p}_0(\Omega)$ and $\wn\in  W^{1,p}_0(\Omega)$ that solve 
\begin{equation}\label{apprp}
\io a(x,\nabla \un)\nabla\phi=\io \frac{|\un|^{p-2}\un}{1+\frac{1}{n}|\un|^{p-1}} E_n(x)\nabla\phi+\io \fn(x)\phi \ \ \ \forall \ \phi\in W^{1,p}_0(\Omega)
\end{equation}
and
\begin{equation}\label{!!appr1}
\io a(x,\nabla\wn)\varphi=\io E_n(x)\frac{|\nabla \wn|^{p-2}\nabla\wn}{1+\frac{1}{n}|\nabla\wn|^{p-1}}\varphi+\io \fn(x)\varphi \ \ \ \forall \ \varphi\in W^{1,p}_0(\Omega),
\end{equation}
respectively, where $E_n(x)$ and $f_n(x)$ are the truncation at level $n\in\mathbb{N}$ of $E(x)$ and $f(x)$.\\ The others steps are obtained in Section \ref{proofof}, while in the following one we provide some preliminary result.
\section{Preliminaries}\label{unoo}

In this section we introduce some preliminary results and tools in order to deal with problems with convection or drift lower order term. In Section \ref{28-10} we give the basic background on the symmetrization technique for elliptic problems introduced in the seminal paper \cite{T}. In Section \ref{mala} we prove the \emph{almost everywhere convergence of the gradients} for the approximating sequences $\{\un\}$ and $\{\wn\}$.
\subsection{Background on symmetrization techniques}\label{28-10}
\begin{prop}\label{29-10}
For $n\in \mathbb{N}$, let $v,v_n:\Omega\to\mathbb{R}$ be measurable functions such that
\[
|v(x)|\le \liminf_{n\to\infty}|v_n(x)| \ \ \ a.e. \ x\in \Omega.
\]
Hence
\[
\overline{v}(s)\le \liminf_{n\to\infty}\overline{v}_n(s) \ \ \ a.e. \ s\in(0,\Omega).
\]
\end{prop} 
\begin{proof}
For the proof see \cite{Gra} Proposition 1.4.5.
\end{proof}

Let us state and prove the following Proposition.
\begin{prop}\label{09-10}
For almost every $s\in(0,|\Omega|)$
\begin{equation}\label{06/10}
A'(\overline{v}(s))\le1 \ \ \ \mbox{and} \ \ \mbox{if} \ \ \overline{v}'(s)\neq0 \ \ \ A'(\overline{v}(s))=\frac{1}{\overline{v}'(s)}.
\end{equation}
\end{prop}
\begin{proof}
Let us consider all the values $s_i$ with $i\in\mathbb{N}$ such that the set
\[
B_i=\{t\in(0,|\Omega|) \ : \ |\overline{v}(t)|=\overline{v}(s_i)\}
\]
has a strictly positive measure. By constriction  every $B_i$ is an half-open proper interval on which $v(s)$ is constant and, since $\overline{v}(s)$ is not increasing, $\overline{B}_i\cap\overline{B}_j=\emptyset$ for $i\neq j$ (this assures us that the $B_i$ are indeed countable). Moreover $\cup_{i\in\mathbb{N}}\overline{B}_i$ is closed and
\[
A'(\overline{v}(s))=0 \ \ \ \forall \ a.e.  \ \ s\in \cup_{i\in\mathbb{N}}\overline{B}_i.
\]
On the other hand setting $K=(0,|\Omega|)\setminus\cup_{i\in\mathbb{N}}\overline{B}_i$ we have that
\[
\forall \ s \in K, \ \ \ |\{|\overline{v}(t)|=\overline{v}(s)\}|=0 \ \ \ \mbox{hence} \ \ \ A(\overline{v}(s))=s.
\]
Since both $\overline{v}(s)$ and $A(s)$ are almost $a.e$ differentiable in $(0,|\Omega|)$ and, since for $a.e. \ s\in K$ it holds true that $\overline{v}'(s)\neq 0$, we have finished.
\end{proof}

Let us state and prove the following useful Lemma (see Lemma 9 of \cite{Okl}).
\begin{lemma}\label{16-8}
For every measurable function $v:\Omega\to\mathbb{R}$, there exists a set valued map $s\to\Omega(s)\subset\Omega$ such that
\begin{equation}\label{night1}
\begin{cases}
|\Omega(s)|=s \ \ \ \mbox{for any} \ s\in[0,|\Omega|],\\
\Omega(s_1)\subset \Omega(s_2) \ \ \ \mbox{whenever} \ \ \ s_1<s_2,\\
\Omega(s)=\{|v|>\overline{v}(s)\} \ \ \ \mbox{if} \  \ \ |\{|v|=\overline{v}(s)\}|=0.
\end{cases}
\end{equation}

\end{lemma}
\begin{rem}\label{okkio}
When we use Lemma \ref{16-8} with $v\equiv\un$ or $\wn$ (see \eqref{apprp} and \eqref{!!appr1} below for the definition of $\un$ and $\wn$) the associated set functions are denoted with $\Omega_{n}(s)$. When we use Lemma \ref{16-8} with $v\equiv|\nabla\un|$ or $|\nabla\wn|$ the associated set function is denoted with $\widetilde{\Omega}_{n}(s)$.
\end{rem}
\begin{proof}
By construction $v(x)$ and $\overline{v}(s)$ are equimeasurable thus
\[
|\{|v(x)|>\overline{v}(s)\}|=|\{|\overline{v}(\tau)|>\overline{v}(s)\}|\le s \le|\{|\overline{v}(\tau)|\ge\overline{v}(s)\}|=|\{|v(x)|\ge\overline{v}(s)\}|.
\]
Since the Lebesgue measure is not atomic there exists $\Omega(s)$ such that
\begin{equation}\label{25-5}
\{|v(x)|>\overline{v}(s)\}\subset \Omega(s)\subset\{|v(x)|\ge\overline{v}(s)\} \ \ \ \mbox{and} \ \ \ |\Omega(s)|=s.
\end{equation}
Of course if $|\{|v|=\overline{v}(s)\}|=0$, then $ \Omega(s)=\{|v(x)|>\overline{v}(s)\}$.\\
\end{proof}
In the next Lemma we define the pseudo rearrangement of a function $g\in\elle 1$ with respect to a measurable function $v(x)$ (see \cite{AT} and \cite{FV}).
\begin{lemma} \label{dife}
Let $v:\Omega\to\mathbb{R}$ a measurable function, $0\le g(x)\in\elle 1$ and $\Omega (s)$ the set valued function associated to $v(x)$ defined in \eqref{night1}. Then 
\begin{equation}\label{night2}
D (s):=\frac{d}{ds}\int_{\Omega (s)}g(x)dx, \  \ \ s\in(0,|\Omega|)
\end{equation}
is well defined and moreover
\begin{equation}\label{24-5bis}
i) \ \ \ \int_0^tD (s)ds=\int_{\Omega (t)}g(x)dx\le \int_0^t\overline{g}(s)ds, \ \ \ t\in(0,|\Omega|)
\end{equation}
\begin{equation}\label{24-5}
ii) \ \ \ D (A (k))(-A' (k))=-\frac{d}{dk}\int_{\{|\un|>k\}}g(x)dx, \ \ \ k>0.
\end{equation}
\end{lemma}
\begin{proof}
Note now that the function defined for $s\in(0,|\Omega|)$ as
\[
s\to\int_{\Omega (s)}g(x)dx
\]
is absolutely continuous in $(0,|\Omega|)$. Thus it is almost everywhere differentiable and, denoting by $D (s)$ its derivative, \eqref{24-5bis} holds true. Reading equation \eqref{24-5bis} for every $s$ such that $s=A (k)$ it follows 
\[
\int_0^{A (k)}D (s)ds=\int_{\Omega (A (k))}g(x)dx=\int_{\{|v|>k\}}g(x)dx,
\]
where we have used that $\Omega (A (k))=\{|v|>\overline{v} (A (k))\}=\{|v|>k\}$. Differentiating with respect to $k$ the previous identity we get \eqref{24-5}.
\end{proof}
The following Lemma assures that the pseudo rearrangement of $g$ has the same summability of $g$. 
\begin{lemma}\label{inl}
Assume that $g\in\elle r$ with $1\le r\le\infty$. Then the function $D (s)$ defined in \eqref{night2} belongs to $L^r((0,|\Omega|)$ and $\|D \|_{L^r(0,|\Omega|)}\le\|g\|_{\elle r}$.\\ Moreover if we assume that $g\in\emme s$ with $1<s<\infty$, then $D $ belongs to $M^s(0,|\Omega|)$.
\end{lemma}
\begin{proof} 
\textbf{Case $g\in\elle r$.} This part of the Lemma has already been proved in \cite{AT} (Lemma 2.2). For the convenience of the reader we provide here the proof. Let us divide the interval $(0,|\Omega|)$ into $i\in\mathbb{N}$ disjoint intervals of the type $(s_{j-1},s_{j})$, for $j=1,\cdots, i$ ,of equal measure $|\Omega|/i$. Let us consider the restriction of $g(x)$ on the set $\Omega (s_{j})\setminus\Omega (s_{j-1})$ and take its decreasing rearrangement in the interval $(s_{j-1},s_{j})$. Repeating this for any $j=1,\cdots, i$ we define a function (up to a zero measure set) on $(0,|\Omega|)$. Clearly this function depends on $i$ and so we call it $ D_{i}(s)$. We stress that by construction the decreasing rearrangement of $ D_{i}(s)$ coincides with the decreasing rearrangement of $g(x)$, thus for any measurable $\omega\subset (0,|\Omega|)$
\begin{equation}\label{tacca}
\int_{\omega}  D_{i}^r(s) ds\le \int_0^{|\omega|}\overline{g}^r(s)ds.
\end{equation}
Hence the sequence $\{ D_{i}^r(s)\}$ is equi-integrable and there exists a function $X\in L^r(0,|\Omega|)$ such that
\[
 D_{i}\rightharpoonup X  \ \ \ \mbox{in} \ L^r(0,|\Omega|) \ \ \ \mbox{as} \ \ \ i\to\infty.
\]
The proof is concluded if we show that $X \equiv D $. Let us define the function
\[
\Phi_{i}(s):=\int_0^s\big( D_{i}(t)-D (t)\big)dt
\] and notice that $\Phi_{i}(0)=\Phi_{i}(|\Omega|)=0$. Thus for any $\varphi(s)\in C^1(0,|\Omega|)$ it results
\begin{multline}	\label{bv}
\int_0^{|\Omega|}\big( D_{i}(s)-D (s)\big) \varphi(s)ds\\=-\int_0^{|\Omega|}\left[\int_0^s\big( D_{i}(t)-D (t)\big)dt\right]d \varphi(s)\le \|\Phi_{i}\|_{L^{\infty}(0,|\Omega|)}\|\varphi'\|_{L^{\infty}(0,|\Omega|)}|\Omega|.
\end{multline}
By construction $\Phi_{i}(s_j)=0$ for any $j=1\cdots i$, since
\[
\int_0^{s_j} D_{i}(t)dt=\sum_{l=1}^{j}\int_{s_{l-1}}^{s_l} D_{i}(t)dt=\sum_{l=1}^{j}\int_{\Omega (s_l)/\Omega (s_{l-1})}g(x)dx
\]
\[
=\int_{\Omega (s_j)}g(x)dx=\int_0^{s_j}D (t)dt.
\]
Hence if $s_{j-1}\le s\le s_j$ we have that
\[
\Phi_{i}(s)=\int_{s_{j-1}}^{s}\big( D_{i}(t)-D(t)\big)dt.
\]
Recalling \eqref{24-5bis} we deduce
\[
-\int_{0}^{|\Omega|/i}\overline{\varphi}(t)dt\le -\int_{s_{j-1}}^{s}D (t)dt\le\int_{s_{j-1}}^{s}\big( D_{i}(t)-D(t)\big)dt\le \int_{s_{j-1}}^{s} D_{i}(t)dt\le \int_{0}^{|\Omega|/i}\overline{\varphi}(t)dt,
\]
that implies the following estimate
\[
|\Phi_{i}(s)|\le \int_{0}^{|\Omega|/i}\overline{\varphi}(t)dt.
\]
Hence the right hand side of \eqref{bv} goes to $0$ as $i$ diverges and
\[
\lim_{i\to\infty}\int_0^{|\Omega|}\big( D_{i}(s)-D (s)\big) \varphi(s)ds= 0,  \ \ \ \forall \ \varphi\in C^1(0,|\Omega|).
\]
Since we already know that $ D_{i}(s)$ admits $X(s)$ as weak limit in $L^r(0,|\Omega|)$, it follows that $X(s)\equiv D (s)$ and we conclude the proof.\\

\textbf{Case $g\in\emme s$.} As in the previous step we can construct a sequence $\{ D_{i}\}$ such that $\overline{D}_{i}(s)=\overline{g}(s)$ for $s\in(0,|\Omega|)$ and
\[
\lim_{i\to\infty}\int_0^{|\Omega|} D_{i}\phi= \int_0^{|\Omega|} D  \phi\ \ \ \forall \ \phi\in L^{\infty}(\Omega).
\]
Take $\phi_A=\chi_A$ with $A\subset (0,|\Omega|)$ and $|A|=s$. We deduce that
\[
\int_0^{|\Omega|} D  \phi_A\le \int_0^s\overline{g} \ \ \mbox{and taking the sup with respect to $A$} \ \ \int_0^{s} \overline{D} \le \int_0^s\overline{g}.
\]
Thus
\[
\overline{D} (s)\le \frac{1}{s}\int_0^{s} \overline{D} \le \frac{1}{s}\int_0^s\overline{g}\le\|g\|_{\emme s}\frac{r}{r-1}s^{-\frac 1r}.
\]

\end{proof}

A key tool in the symmetrization process introduced in \cite{T} is given by the following Proposition.
\begin{prop}
For any $v\in W^{1,p}_0(\Omega)$ and for any $s\in\mathbb{R}$
\begin{equation}\label{talenti}
 \sigma_N\le  A(s)^{\frac{1}{N}-1}\big(-A'(s)\big)^{\frac{1}{p'}}\left(-\frac{d}{ds}\int_{A(s)}|\nabla v|^p\right)^{\frac{1}{p}},
\end{equation}
where $\sigma_N=N\omega_N^{\frac{1}{N}}$ and $\omega_N$ is the volume of the unitary ball in dimension $N$.
\end{prop}
\begin{proof}
See pages 711 and 712 of \cite{T}.
\end{proof}

 The next Lemma is used to establish the membership to Lorentz spaces of some integral quantities. 
\begin{lemma}\label{trickint}
Let $r:(0,+\infi)\to(0,+\infi)$ be a decreasing function and let us define for $\beta\ge0$ and $\delta\neq1$
\begin{equation}
R_{\delta}(t):=\begin{cases}
\int_0^t s^{\beta}r(s) ds \ \ \ \mbox{if} \ \ \ \delta<1\\
\int_t^{+\infi} s^{\beta}r(s) ds \ \ \ \mbox{if} \ \ \ \delta>1.
\end{cases}
\end{equation}
Then for every $\lambda>0$ there exists $C=C(\beta,\delta,\lambda)$ such that
\[
\int_0^{\infi}\left(\frac{R_{\delta}(t)}{t}\right)^{\lambda}t^{\delta\lambda}\frac{dt}{t}\le C\int_0^{\infi} r(t)^{\lambda}t^{\lambda(\beta+\delta)}\frac{dt}{t}.
\]
\end{lemma}
\begin{proof}
For the proof see \cite{AFT} Lemma 2.1.
\end{proof}

\subsection{Others useful results}\label{mala}
 
In this Section we prove the almost everywhere convergence of the gradients of $\{\un\}$ and $\{\wn\}$.
\begin{lemma}\label{22-3}
Let  $\{\un\}\subset W^{1,p}_0(\Omega)$ be the sequence of \emph{approximating} solutions of \eqref{apprp}. Assume $f\in\elle1$, $E\in\left(\elle{p'}\right)^N$ and  moreover that there exists $u\in W^{1,s}_0(\Omega)$ with $s\ge1$ such that up to a subsequence $\un\rightharpoonup u$ in $W^{1,s}_0(\Omega)$. Hence, up to a further subsequence,
\begin{equation}\label{aecon}
\nabla \un\to \nabla u \ \ \ a.e. \ \ \ \mbox{in} \ \Omega.
\end{equation}
\end{lemma}
\begin{rem}
Notice that the assumption for $E(x)$ of the Lemma above is more general that \eqref{marp}.
\end{rem}
\begin{proof}Taking $T_k(\un)$ as test function in \eqref{apprp} and using Young inequality it follows that for any $\epsilon>0$
\[
\alpha\io|\nabla T_k(\un)|^p\le C_{\epsilon}k^p\io|E|^{p'} +\epsilon\io|\nabla T_k(\un)|^p+k\io|f|,
\]
with $C_{\epsilon}=\epsilon^{-\frac{1}{p-1}}$. Thanks to the previous estimate we deduce that for every $k>0$  
\begin{equation}\label{trunc}
|\nabla T_k(u)|\in \elle p \ \ \ \mbox{and} \ \ \ T_k(\un)\to T_k(u) \ \ \ \mbox{weakly in } W^{1,p}_0(\Omega).
\end{equation}
In order to prove \eqref{aecon} let us define for $k>0$ fixed 
\[
I_n^k(x)= [a(x,\nabla T_k(\un))-a(x,\nabla T_k(u))]\nabla(T_k(\un)-T_k(u))
\]
and consider, for $0<\theta<1$ and $0<h<k$,
\[
\io I_n^k(x)^{\theta}dx=\int_{\{|T_k(\un)-T_k(u)|>h\}}I_n^k(x)^{\theta}dx+\int_{\{|T_k(\un)-T_k(u)|\le h\}}I_n^k(x)^{\theta}dx
\]
\[
\le \left(\io I_n^k(x)dx\right)^{\theta}|\{|T_k(\un)-T_k(u)|>h\}|^{1-\theta}+\left(\int_{\{|T_k(\un)-T_k(u)|\le h\}}I_n^k(x)dx\right)^{\theta}|\Omega|^{1-\theta}.
\]
Note that, for every fixed $h$, the first term in the right hand side above goes to zero as $n\to\infi$ because of \eqref{trunc} and thanks to the convergence in measure of $T_k(\un)$. We claim that also the second term converge to zero taking the limit at first with respect to $n\to \infi$ and then with respect to $h\to 0$. Once this claim is proved, it follows that
\[
\lim_{n\to\infi}\io I_n^k(x)^{\theta}dx=0,
\]
from which we deduce, like in \cite{BCa}, that $\nabla T_k(\un)$ almost everywhere converges to
$\nabla T_k(u)$ for every $k> 0$. An this is enough to infer \eqref{aecon} as in \cite{Po}.\\ In order to prove the claim let us take $T_h(\un-T_k(u))$, with $0<h<k$, as a test function in \eqref{apprp}. After simple manipulations we obtain that
\[
-\int_{\{|\un-T_k(u)|<h\}}a(x,\nabla G_k(\un))\nabla T_k(u)+\io a(x,\nabla T_k(\un))\nabla T_h(T_k(\un)-T_k(u))
\]
\[
 \le h\int|f|+\io\frac{|\un|^{p-2}\un }{1+\frac{1}{n}|\un|^{p-1}}E_n(x)\nabla T_h(\un-T_k(u))
\]
and also that
\[
0\le\int_{\{|T_k(\un)-T_k(u)|\le h\}}I_n^k(x)dx=\io \big[a(x,\nabla T_k(\un))- a(x,\nabla T_k(u))\big]\nabla T_h(T_k(\un)-T_k(u)) 
\]
\[
\le h\int|f|+\io\frac{|\un|^{p-2}\un }{1+\frac{1}{n}|\un|^{p-1}}E_n(x)\nabla T_h(\un-T_k(u))
\]
\[
+\int_{\{|\un|>k\}\cap\{|\un-T_k(u)|<h\}}a(x,\nabla\un)\nabla T_k(u)-\io a(x,\nabla T_k(u))\nabla T_h(T_k(\un)-T_k(u)).
\]
Noticing that $\{|\un-T_k(u)|<h\}\subset \{|\un|\le h+k\}\subset \{|\un|\le 2k\}$, that the sequence $\{|a(x,\nabla T_{2k}(\un))|\}$ is bounded in $\elle{p'}$ and recalling \eqref{trunc}, we can pass to the limit with respect to $n\to \infi$ into the previous inequality and obtain
\[
\limsup_{n\to\infi}\int_{\{|T_k(\un)-T_k(u)|\le h\}}I_n^k(x)dx\le\io|u|^{p-2}u\,E\nabla T_h(G_k(u))+h\int|f|
\]
\[
+\int_{\{k<|u|<k+h\}}\Psi_k\nabla T_k(u).
\]
where $\Psi_k\in\left(\elle{p'}\right)^N$ is the weak limit of $a(x,\nabla T_{2k}(\un))$.
Letting $h\to0$ we prove the claim and conclude the proof of the Lemma.
\end{proof}

\begin{lemma}\label{aenodiv}
Let $\{\wn\}\subset W^{1,p}_0(\Omega)$ be the sequence of approximating solution of \eqref{!!appr1}. Assume $f\in\elle 1$, $|E|\in\emme N$ and  moreover that there exists $w\in W^{1,s}_0(\Omega)$ with $s>\frac{(p-1)N}{N-1}$ such that up to a subsequence $\wn\rightharpoonup w$ in $W^{1,s}_0(\Omega)$. Hence, up to a further subsequence,
\begin{equation}\label{aebis}
\nabla \wn\to \nabla w \ \ \ a.e. \ \ \ \mbox{in} \ \Omega.
\end{equation}
\end{lemma}
\begin{proof}By hypothesis the sequence $\{|\nabla\wn|^{p-1}\}$ is bounded in ${\elle r}$ with $r=\frac{s}{p-1}>\frac{N}{N-1}$ and moreover $r'<N$. Hence taking $T_k(\wn)$ as a test function in \eqref{!!appr1}, we obtain
\[
\alpha\io|\nabla T_k(\wn)|^p\le k\left[\io|f|+\io|E_n(x)||\nabla\wn|^{p-1}\right]\]\[ \le k \left[\|f\|_{\elle 1}+\|E\|_{\elle {r'}}\||\nabla\wn|^{p-1}\|_{\elle {r}}\right],
\]
that implies 
\[
T_k(\wn)\rightharpoonup T_k(w) \ \ \ \mbox{in} \ \ \ W^{1,p}_0(\Omega) \ \ \ \mbox{for any} \ \ \ k>0.
\]

Notice that we are in the same situation of Lemma \ref{22-3} above. Thus we conclude the proof if we show that
\[
\limsup_{n\to\infi}\io \big[a(x,\nabla T_k(\wn))- a(x,\nabla T_k(w))\big]\nabla T_h(\wn-T_k(w))=0.
\]
As before let us thus choose $T_h(\wn-T_k(w))$, with $0<h<k$, as test function in \eqref{!!appr1}. Manipulating the resulting equation, we obtain
\[
\io \big[a(x,\nabla T_k(\wn))- a(x,\nabla T_k(w))\big]\nabla T_h(\wn-T_k(w))\] \[\le h\left[\|f\|_{\elle 1}+\|E\|_{\elle {r'}}\||\nabla\wn|^{p-1}\|_{\elle {r}}\right]
+\int_{\{|\wn|>k\}\cap\{|\wn-T_k(w)|<h\}}a(x,\nabla\wn)\nabla T_k(w).
\]
\[
-\io a(x,\nabla T_k(w))\nabla T_h(\wn-T_k(w)).
\]
Noticing that $\{|\wn-T_k(w)|<h\}\subset \{|\wn|\le h+k\}\subset \{|\wn|\le 2k\}$ we can pas to the limit with respect to $n\to \infi$ and obtain
\[
\limsup_{n\to\infi}\io a(x,\nabla T_k(\wn))\nabla T_h(\wn-T_k(w))\le C h
+\int_{\{k<|w|<k+h\}}\Psi_k\nabla T_k(w),
\]
where $\Psi_k\in\left(\elle{p'}\right)^N$ is the weak limit of $a(x,\nabla T_{2k}(\wn))$. Letting $h\to0$ we  conclude the proof of the Lemma.
\end{proof}
\begin{lemma}\label{gron}
Given the function $\lambda,\ \gamma, \ \varphi, \rho$ defined in $(0,+\infty)$, suppose that $\lambda,\ \gamma\ge 0$ and that $\lambda\gamma, \ \lambda\varphi$ and $\lambda\rho$ belong to $L^1(0,\infty)$. If for almost every $t\ge0$ we have
\[
\varphi(t)\le \rho(t)+\gamma(t)\int_t^{+\infty}\lambda(\tau)\varphi(\tau)d\tau,
\]
then for almost every $t\ge0$
\[
\varphi(t)\le \rho(t)+\gamma(t)\int_t^{+\infty}\rho(t)\lambda(\tau)e^{\int_t^{\tau}\lambda(s)\gamma(s)ds}d\tau.
\]
\end{lemma}
\begin{proof}
See \cite{HA} Lemma 6.1.
\end{proof}

\section{Proof of the results}\label{proofof}
 
\subsection{Convection term}
We need three preliminary Lemmas. The first one is devoted to the achievement of a point-wise estimate for the decreasing rearrangement of $\un$, the solution of \eqref{apprp}, the second Lemma gives the estimate relative to the decreasing rearrangement of $\nabla\un$, while the third one provides the required Lorentz bounds for the sequences $\{\un\}$ and $\{|\nabla\un|\}$. 
\begin{lemma}\label{regup} Let us assume \eqref{ll5} and \eqref{marp}. 
For any $n\in\mathbb{N}$, let $\un$ be the solution of \eqref{apprp} and denote with $\overline{u}_n$ its decreasing rearrangement. It follows that
\begin{equation}\label{10-03p}
\overline{u}_n(t)\le \overline{v}(t):= \frac{C}{t^{\gamma}}\int_t^{|\Omega|}s^{\frac{p'}{N}+\gamma-1}\tilde{f}(s)^{\frac{1}{p-1}}ds,
\end{equation}
where $C=C(N,\alpha,p,E,m)$ and $\gamma<\frac{N-pm}{(p-1)Nm}$.
\end{lemma}
\begin{rem}
In order to better understand \eqref{10-03p} let us set, in the special case $p=2$ and with a slight abuse of notation,
\[
v(x)=\overline{v}(\omega_N|x|^N)=\frac{C}{\left(\omega_N|x|^N\right)^{\gamma}}\int_{\omega_N|x|^N}^{|\Omega|}s^{\frac{2}{N}-1+\gamma}\tilde{f}(s)ds|
\]
and notice that it solves
\begin{equation} \label{simmetrizzato}
\begin{cases}\displaystyle
-\Delta v=\gamma NC\mbox{div}\left(v\,\frac{x}{|x|^2}\right)+\omega_N^{\frac2N}N^2C\tilde{f}(\omega_n|x|^N) & \mbox{in $B_{\Omega}$,} \\
\hfill v = 0 \hfill & \mbox{on $\partial B_{\Omega}$,} 
\end{cases}
\end{equation}
where $B_{\Omega}$ is the ball centered at the origin sucht that $|B_{\Omega}|=|\Omega|$ and $C$ and $\gamma$ are the constant of Lemma \eqref{regup}. Thus inequality \eqref{10-03p} provides the already mentioned comparison between the rearrangements of the solution of the original problem and the symmetrized one.\\
\end{rem}

\begin{proof} We apply to our contest the approach of \cite{DPbis}. Since $\un\in W^{1,p}_0(\Omega)$, we are allowed to take $\frac{T_h(\Gk(\un))}{h}$ with $h>0$ and $k\ge0$ as test function in \eqref{apprp}, so that we get
\begin{equation}\label{deltahp}
\frac{\alpha}{h}\int_{\{k<|\un|<k+h\}}|\nabla\un|^p\le \int_{\{|u_n|>k\}}|f|+\frac{(k+h)^{p-1}}{h}\int_{\{k<|\un|<k+h\}}|E||\nabla \un|.
\end{equation}
Applying H\"older inequality to the last integral in the right hand side above and letting $h$ go to zero, we obtain
\begin{equation}\label{art1}
-\frac{d}{dk}\int_{\{|\un|>k\}}|\nabla \un|^p\le \frac{\int_{\{|\un|>k\}}|f|}{\alpha}+\frac{k^{p-1}}{\alpha}\left(-\frac{d}{dk}\int_{\{|\un|>k\}}|\nabla \un|^p\right)^{\frac{1}{p}}\left(-\frac{d}{dk}\int_{\{|\un|>k\}}|E|^{p'}\right)^{\frac{1}{p'}}.
\end{equation}
Let us set for any $n\in\mathbb{N}$ and $k>0$
\[
A_n(k)=|\{|\un|>k\}|,
\]
namely $A_n(k)$ is the distribution function of $\un$. Let us moreover introduce the pseudo rearrangements of $|\mathcal{F}|^2$ and $|\mathcal{E}|^2$ with respect to $\un$ (see \eqref{night2} for the definition)
\[
D_{1,n}(s):=\frac{d}{ds}\int_{\Omega_n(s)}|\mathcal{F}(x)|^pdx \ \ \ \mbox{and} \ \ \ D_{2,n}(s):=\frac{d}{ds}\int_{\Omega_n(s)}|\mathcal{E}(x)|^pdx, \ \ \ \mbox{with} \ \ \ s\in(0,|\Omega|).
\]
Thanks to \ref{24-5} we have that for $k>0$
\[
D_{1,n}(A_n(k))(-A_n'(k))=-\frac{d}{dk}\int_{\{|\un|>k\}}|\mathcal{F}|^{p'} \  \ \mbox{and} \ \  D_{2,n}(A_n(k))(-A_n'(k))=-\frac{d}{dk}\int_{\{|\un|>k\}}|\mathcal{E}|^{p'}.
\]
Setting $\mathcal{D}_n(s)=D_{1,n}(s)+D_{2,n}(s)$, to have a more compact notation and using \eqref{talenti}, inequality \eqref{art1} becomes
\begin{multline} \label{12-02p}
\left(-\frac{d}{dk}\int_{\{|\un|>k\}}|\nabla \un|^p\right)^{\frac{1}{p'}}\\\le \frac{A_n(k)^{\left(\frac{1}{N}-1\right)}}{\alpha\sigma_N} \int_{\{|\un|>k\}} |f|\big(-A'_n(k)\big)^{\frac{1}{p'}}+\frac{k^{p-1}}{\alpha}\mathcal{D}_n(A_n(k))^{\frac{1}{p'}}\big(-A'_n(k)\big)^{\frac{1}{p'}},
\end{multline}
that can be rewritten, using once more \eqref{talenti}, as
\[
1 \le\left[ \frac{A_n(k)^{p\left(\frac{1}{N}-1\right)}}{\alpha\sigma_N^p} \int_{\{|\un|>k\}} |f|+\frac{k^{p-1}}{\alpha\sigma_N^{p-1}}\mathcal{D}_n(A_n(k))^{\frac{1}{p'}}A_n(k)^{\left(\frac{1}{N}-1\right)(p-1)}\right](-A'_n(k))^{p-1}.
\]
Thanks to the definition of decreasing rearrangement and using Proposition \ref{09-10} in Section \ref{unoo}, it results
\begin{equation}\label{SB}
-\frac{d}{ds}\overline{u}_n(s)\le\left[\frac{s^{p(\frac{1}{N}-1)}}{\alpha\sigma_N^p}\int_0^{s}\bar{f}+\frac{1}{\alpha\sigma_N^{p-1}}\mathcal{D}_n(s)^{\frac{1}{p'}}s^{\left(\frac1N-1\right)(p-1)}\overline{u}_n^{p-1}(s)\right]^{\frac{1}{p-1}}
\end{equation}
\[ \le C_{\delta} s^{p'(\frac{1}{N}-1)}\left(\int_0^{s}\bar{f}(\tau)d\tau\right)^{\frac{1}{p-1}}+\frac{\delta}{\alpha^{\frac{1}{p-1}}\sigma_N}\mathcal{D}_n(s)^{\frac1p}s^{\frac{1}{N}-1}\overline{u}_n(s),
\]
where $\delta>1$ is such that
\[
\gamma=\frac{\delta B}{\alpha^{\frac{1}{p-1}}\sigma_N}<\frac{N-pm}{(p-1)Nm}.
\]
This is possible thanks to assumption \eqref{marp}. Defining the auxiliary function
\[
R_n(s)=e^{\frac{\gamma}{B}\int_t^s\mathcal{D}_n(\tau)^{\frac1p}\tau^{\frac{1}{N}-1}d\tau}, \ \ \ \mbox{with} \ \ \ t<s, 
\]
we finally deduce that
\[
-\frac{d}{ds}\big(R(s)\overline{u}_n(s)\big)\le C s^{p'(\frac{1}{N}-1)}R_n(s)\left(\int_0^{s}\bar{f}(\tau)d\tau\right)^{\frac{1}{p-1}}.
\]

In order to estimate $R_n(s)$ we recall the definition of $\mathcal{D}_n$ and Lemma \ref{inl}. It results that
\[
\int_t^sD_{1,n}(\tau)^{\frac1p}\tau^{\frac1N-1}d\tau\le \|\mathcal{F}\|_{\elle{\infty}}\frac{N}{\alpha^{\frac{1}{p-1}}\sigma_N}|\Omega|^{\frac1N}
\]
and, using Young Inequality, integration by parts and assumption \eqref{marp}, that
\begin{equation*}
\int_t^sD_{2,n}(\tau)^{\frac1p}\tau^{\frac1N-1}d\tau\le \frac{1}{pB^{p-1}}\int_t^sD_{2,n}(\tau)\tau^{\frac pN-1}d\tau+\frac{B}{p'}\int_t^s\frac{1}{\tau}d\tau\end{equation*}
\[
\le \frac{1}{pB^{p-1}}\left[s^{\frac{p}{N}-1}\int_0^s\overline{\mathcal{E}}^{p'}-t^{\frac{p}{N}-1}\int_0^t\overline{\mathcal{E}}^{p'}-\frac{p-N}{N}\int_t^s\tau^{\frac{p}{N}-2}\int_0^{\tau}\overline{\mathcal{E}}^{p'}d\tau\right]+\frac{B}{p'}\log\left(\frac{s}{t}\right)
\]
\[
\le  \frac{NB}{p(N-p)}+B\log\left(\frac{s}{t}\right).
\]
Thus we have that
\[
R_n(s)=e^{\frac{\gamma}{B}\int_t^s\mathcal{D}_n(\tau)^{\frac1p}\tau^{\frac{1}{N}-1}d\tau}\le C\left(\frac{s}{t}\right)^{\gamma}.
\]
Integrating between $t$ and $|\Omega|$ and recalling that by definition of both $\overline{u_n}(|\Omega|)=0$ and $R(t)=1$, we get
\[
\overline{u}_n(t)=-R(|\Omega|)\overline{u}_n(|\Omega|)+R(t)\overline{u}_n(t)\le \frac{C_1}{t^{\gamma}}\int_t^{|\Omega|}s^{p'(\frac{1}{N}-1)+\gamma}\left(\int_0^{s}\bar{f}(\tau)d\tau\right)^{\frac{1}{p-1}}ds.
\]
Thus the proof of the Lemma is concluded.
\end{proof}
The next Lemma is the core of our main result and  provides the estimate relative to the decreasing rearrangement of $\nabla\un$. 
\begin{lemma}\label{regdup} Let us assume \eqref{ll5} and \eqref{marp}.
Let $\overline{|\nabla \un|}$ be the decreasing rearrangement of $|\nabla\un|$. There exists $C=C(N,\alpha,p,E,m)$ such that 
\begin{multline}\label{gradexp}
\frac{1}{s}\int_0^{s}\overline{|\nabla \un|}^{p-1}\le C\left[\frac{1}{s}\int_0^s\big(v(t)^{p-1}\mathcal{D}_n(t)^{\frac{1}{p'}}+\tilde{f}(t)t^{\frac{1}{N}}\big)dt\right. \\ \left.+\left(\frac{1}{s}\int_s^{|\Omega|}\big(v(t)^{p}\mathcal{D}_n(t)+\tilde{f}(t)^{p'}t^{\frac{p'}{N}}\big)dt\right)^{\frac{1}{p'}}\right],
\end{multline}
where $v(t)$ is defined in \eqref{10-03p}.
\end{lemma}
\begin{proof}
Taking advantage of Lemma \ref{16-8} (see Remark \ref{okkio}), it follows that
\[
\int_0^s\overline{|\nabla \un|}^{p-1}d\tau =\int_{\widetilde{\Omega}_s}|\nabla \un|^{p-1}dx\]
\[= \int_{\widetilde{\Omega}_s\cap\{|\un|>\overline{u}_n(s)\}}|\nabla\un|^{p-1}dx+\int_{\widetilde{\Omega}_s\cap\{|\un|\le\overline{u}_n(s)\}}|\nabla \un|^{p-1}dx\]
\[
\le\int_{\{|\un|>\overline{u}_n(s)\}}|\nabla \un|^{p-1}dx+
 \left(\int_{\{|\un|\le\overline{u}_n(s)\}}|\nabla \un|^{p}dx\right)^{\frac{1}{p'}}|\widetilde{\Omega}_s|^{\frac{1}{p}}\le I_1(s)+I_2^{\frac{1}{p'}}(s)s^{\frac{1}{p}}.
\]
As far as $I_2$ is concerned we infer from \eqref{12-02p} that
\begin{equation}\label{composition}
\frac{d}{ds}\int_{\{|\un|>\overline{u}_n(s)\}}|\nabla \un|^p=\left.\frac{d}{dk}\int_{\{|\un|>k\}}|\nabla \un|^p\right|_{k=\overline{u}_n(s)}\frac{d}{ds}\overline{u}_n(s)\end{equation}\[\le C\left[ \overline{u}_n(s)^p\mathcal{D}_n(s)+s^{\frac{p'}{N}}\tilde{f}(s)^{p'}\right].
\]
Integrating between $s$ and $|\Omega|$, we get
\begin{multline*}
I_2=\int_{\{|\un|\le \overline{u}_n(s)\}}|\nabla \un|^p=-\int_{\{|\un|>\overline{u}_n(s)\}}|\nabla \un|^p+\io|\nabla \un|^p\\
\le C\left[\int_{s}^{|\Omega|}\overline{u}_n(t)^p \mathcal{D}_n(t)+ t^{\frac{p'}{N}}\tilde{f}(t)^{p'}dt\right].
\end{multline*}

In order to estimate $I_1$ notice that 
\[
\int_{\{\overline{u}_n(s)\le|\un|<\overline{u}_n(s+h)\}}|\nabla\un|^{p-1}\]\[\le \left(\int_{\{\overline{u}_n(s)\le|\un|<\overline{u}_n(s+h)\}}|\nabla\un|^{p}\right)^{\frac1{p'}}|\{\overline{u}_n(s)\le|\un|<\overline{u}_n(s+h)\}|^{\frac1p}.
\]
Passing to the limit as $h\to 0$ and recalling that $|\{|\un|>\overline{u}_n(s)\}|'\le1$ thanks to Lemma \ref{09-10}, we obtain that
\[
\frac{d}{ds}\int_{\{|\un|>\overline{u}_n(s)\}}|\nabla \un|^{p-1}\le\left(\frac{d}{ds}\int_{\{|\un|>\overline{u}_n(s)\}}|\nabla \un|^{p}\right)^{\frac1{p'}}\le C\left(\overline{u}_n(s)^{p-1}\mathcal{D}_n^{\frac{1}{p'}}(s)+\tilde{f}(s)s^{\frac{1}{N}}\right).
\]

Hence we have the following estimate for $I_1$
\[
I_1\le C\int_0^s\left(\overline{u}_n(t)^{p-1}\mathcal{D}_n^{\frac{1}{p'}}(t)+\tilde{f}(t)t^{\frac{1}{N}}\right)dt.
\]
Putting together the obtained information for $I_1$ and $I_2$ we recover \eqref{gradexp}.
\end{proof}
The previous estimates on the decreasing rearrangements of $\un$ and $\nabla\un$ allow us to obtain the following Lorentz estimates in function of the Lorentz summability of the datum $f$.
\begin{lemma}\label{dupap} Let $\{\un\}$ be the sequence of solutions of \eqref{apprp}.\\
(i) If $f\in\lor{m}{q}$ with $1< m< (p^*)'$ and $0<q\le\infi$, then 
\[
\|\un\|_{\lor{\frac{(p-1)Nm}{N-pm}}{(p-1)q}}\le C\|f\|_{\lor mq} \ \ \ \mbox{and} \ \ \ \|\nabla \un\|_{\lor{\frac{(p-1)Nm}{N-m}}{(p-1)q}}\le C\|f\|_{\lor mq}.
\]
(ii) If $f\in\lorr{1}{q}$ with $0<q\le\infty$
\[
\|\un\|_{\lor{\frac{(p-1)N}{N-p}}{(p-1)q}}\le C\|f\|_{\lorr{1}{q}}\ \ \ \mbox{and}  \ \ \ \|\nabla\un\|_{\lor{\frac{(p-1)N}{N-1}}{(p-1)q}}\le C\|f\|_{f\in\lorr{1}{q}}.
\]
(iii) If $f\in\elle1$, then
\[
\|\un\|_{\lor{\frac{(p-1)N}{N-p}}{\infi}}\le C\|f\|_{\elle1}\ \ \ \mbox{and}  \ \ \ \|\nabla\un\|_{\lor{\frac{(p-1)N}{N-1}}{\infi}}\le C\|f\|_{\elle1}.
\]
\end{lemma}
\begin{proof} 
\textbf{\emph{(i)}.} Le us start with the $f\in \lor mq$ with $1\le m<\left(p^*\right)'$ and $0<q<\infi$.
\emph{Estimate for $\{\un\}$.} 
Using \eqref{10-03p} we get
\begin{equation}\label{stop}
\||u|^{p-1}\|^q_{L^{\frac{Nm}{N-pm},q}(\Omega)}=\int_0^{+\infi}t^{\frac{q(N-pm)}{Nm}}\overline{u}(t)^{(p-1)q}\frac{dt}{t}\end{equation}
\[\le C\int_0^{+\infi}t^{\frac{q(N-pm)}{Nm}-\gamma (p-1)q}\left(\int_t^{|\Omega|}s^{\frac{p'}{N}+\gamma-1}\tilde{f}(s)^{\frac{1}{p-1}}ds\right)^{(p-1)q}\frac{dt}{t}\]
\[= C \int_0^{\infi}t^{\frac{q(N-pm)}{Nm}-\gamma (p-1)q+(p-1)q}\left(\frac{\int_t^{|\Omega|}s^{\frac{p'}{N}+\gamma-1}\tilde{f}^{\frac{1}{p-1}}}{t}\right)^{(p-1)q}\frac{dt}{t} \le  C\int_0^{\infi}t^{\frac{q}{m}}\overline{f}^q\frac{dt}{t} ,
\]
where the last inequality comes from Lemma \ref{trickint} with $\delta=\frac{N-pm}{Nm(p-1)}-\gamma+1>1$, thanks to the choice of $\gamma$.\\ In the case $q=+\infi$, we obtain directly from \eqref{10-03p} that
\[
\overline{u}(s)\le \frac{C}{s^{\frac{N-pm}{Nm(p-1)}}}\|f\|_{L^{m,\infi}(\Omega)}.
\]
\emph{Estimate for $\{\nabla \un\}$.} Thank to Lemma \ref{inl} estimate \eqref{gradexp} can be rewritten as
\begin{equation}\label{10:01}
\frac{1}{s}\int_0^{s}\overline{|\nabla \un|}^{p-1}\le C\left[\frac{1}{s}\int_0^s\big(v(t)^{p-1}t^{-\frac{p-1}{N}}+\tilde{f}t^{\frac{1}{N}}\big)dt+\left(\frac{1}{s}\int_s^{|\Omega|}\big(v(t)^{p}t^{-\frac pN}+\tilde{f}^{p'}t^{\frac{p'}{N}}\big)dt\right)^{\frac{1}{p'}}\right].
\end{equation}

In order to prove the membership of the four terms above to $L^{m^*,q}(\Omega)$ we use Lemma \ref{trickint} 
\[
\int_0^{\infi}s^{\frac{q}{m^*}}\left(\frac{1}{s}\int_0^s v(t)^{p-1}t^{-\frac{p-1}{N}}dt\right)^q\frac{ds}{s}\le C \int_0^{\infi}s^{\frac{q(N-pm)}{Nm}} v(s)^{(p-1)q}\frac{ds}{s}<\infi,
\]
\[
\int_0^{\infi}s^{\frac{q}{m^*}}\left(\frac{1}{s}\int_0^s\tilde{f}(t)t^{\frac{1}{N}}dt\right)^q\frac{ds}{s}\le C \int_0^{\infi}s^{\frac{q}{m}}\tilde{f}(s)^q\frac{ds}{s}<\infi,
\]
where we take  $\delta=\frac{1}{m^*}<1$, and
\[\int_0^{+\infi}s^{\frac{q}{m^*}}\left(\frac{1}{s}\int_{s}^{|\Omega|} \tilde{f}(t)^{p'}t^{\frac{p'}{N}}dt\right)^{\frac{q}{p'}}\frac{ds}{s}\le \int_0^{+\infi}s^{\frac{q}{m}}\tilde{f}(s)^{q}\frac{ds}{s}<\infi,
\]
\[+\int_0^{+\infi}s^{\frac{q}{m^*}}\left(\frac{1}{s}\int_{s}^{|\Omega|}\frac{ v(t)^p}{t^{\frac{p}{N}}}dt\right)^{\frac{q}{p'}}\frac{ds}{s}\le
\int_0^{+\infi}s^{\frac{q(N-pm)}{Nm}} v(s)^{(p-1)q}\frac{ds}{s}<\infi.
\]
where we take $\delta=\frac{p'}{m^*}<1$ (recall that $m<\left(p^*\right)'$). Hence we have that
\[
\|\nabla\un\|_{\lor{\frac{(p-1)Nm}{N-m}}{(p-1)q}}^q\le\int_0^{\infty}\tau^{\frac{q}{m^*}}\left(\frac{1}{s}\int_0^\tau\overline{|\nabla\wn|}^{p-1}(t)dt\right)^q\frac{d\tau}{\tau}\le C\|f\|_{\lor{m}{q}}.
\]
In the case $q=\infi$ we obtain by direct calculation from \eqref{10:01} that
\[
\|\nabla u\|_{L^{\frac{(p-1)Nm}{N-m},\infi}(\Omega)}\le C\|f\|_{L^{m,\infi}(\Omega)}.
\]
\textbf{\emph{(ii)}.} It follows exactly the same argument of \emph{(i)}.\\
\textbf{\emph{(iii)}.} Inequality \eqref{10-03p} becomes 
\[
\overline{u}_n(t)\le v(t)\le C \|f\|_{\elle1}\frac{1}{t^{\gamma}}\int_t^{|\Omega|}s^{p'\left(\frac{1}{N}-1\right)+\gamma}ds \le C\|f\|_{\elle1} t^{-\frac{N-p}{(p-1)N}},
\]
where we have used that $p'\left(\frac{1}{N}-1\right)+\gamma+1<0$. On the other hand we have that
\[
\frac{1}{t}\int_0^t\overline{|\nabla\un|}^{p-1}\le C \|f\|_{\elle1}\left[\frac{1}{t}\int_0^t s^{\frac{1}{N}-1}ds+\right(\frac{1}{t}\int_0^t s^{p'\left(\frac{1}{N}-1\right)}ds\left)^{\frac{1}{p'}}\right]\le C \|f\|_{\elle1} t^{-\frac{N-1}{N}},
\]
and thus the proof is concluded.
\end{proof}
\begin{rem}\label{nirvana}
In order to show that assumption \eqref{marp} is sharp for Theorems \ref{teohighp} and \ref{teolowp} to hold, let us consider the solution of the symmetrized problem \eqref{simmetrizzato} in the simple case $p=2$ and $f\in\emme m$ with $1<m<(2^*)'$. For this value of $p$ notice that $\gamma=\gamma(B)=\frac{B}{\alpha N\omega_N^{\frac1N}}$ (see \eqref{SB}). If we take now $B>0$ so that $\frac{N-2m}{Nm}<\gamma(B)<1$ (\eqref{marp} is not satisfied) it follows that
\[
\overline{v}(t)\le \frac{C\|f\|_{\emme m}}{t^{\gamma}}\int_t^{|\Omega|}s^{\frac{2}{N}+\gamma-1-\frac1m}ds\le \frac{\tilde{C}}{t^{\gamma}}.
\]
Thus $u\in \emme{\frac{1}{\gamma(B)}}$. Moreover, if $\frac{N-2m}{Nm}<\gamma(B)<\frac{N-1}{N}$, we deduce by direct computation of the gradient that
\[
\overline{|\nabla v|}(t)\le \frac{C}{t^{\gamma(B)+\frac1N}}.
\]
Hence for any $1<r<\frac{Nm}{N-m}$ we can chose $B$ in order to have $|\nabla v|\in\emme{r}$. In the borderline case $\gamma(B)=\frac{N-2m}{Nm}$ estimates with logarithmic corrections are obtained.\\
This argument shows that if \eqref{marp} is not satisfied the standard relation between the regularity of the data and the solution is lost (for a more detailed description of this fact see \cite{BO}).
\end{rem}

Now we are in the position of proving Theorems \ref{teohighp} and \ref{teolowp}. We start from the latter.

\begin{proof}[Proof of Theorem \ref{teolowp}.]
\textbf{Case \emph{(i)}.} Let us start with the case $p>2-\frac{1}{N}$ and $f\in\elle1$. From Lemma \ref{dupap} we deduce that the sequence $\{|\nabla\un|\}$ is bounded in $\lor{\frac{(p-1)N}{N-1}}{\infty}$ and, in turn, in $\elle r$ for any $1<r<\frac{(p-1)N}{N-1}$. Hence there exists $u\in W^{1,r}_0(\Omega)$ such that  $\un\rightharpoonup u$ in $W^{1,r}_0(\Omega)$. Thanks to the almost everywhere convergence of the gradients proved in Lemma \ref{22-3},  we infer that 
\[
\nabla\un\to \nabla u \ \ \ \mbox{in} \ \ \ L^{\frac{r}{p-1}}(\Omega).
\] 
Observing that it is possible to choose $r$ such that $\frac{r}{p-1}>1$, it follows that
\begin{equation}\label{9/2/2019}
|\nabla \un|^{p-2}\nabla \un  \to |\nabla u|^{p-2}\nabla u  \ \ \ \mbox{in} \ \ \ \elle1.
\end{equation}
Thus we can pass to the limit, as $n\to\infty$, in the left hand side of \eqref{apprp} for every $\phi\in C^1_0(\Omega)$. In order to handle the lower order term, notice that for every measurable $\omega\subset\Omega$ it follows that
\begin{equation}\label{lowlim}
 \int_\omega|\un|^{p-1}|E_n|\le \int_0^{|\omega|} v^{p-1}(t)t^{-\frac{p-1}{N}}dt\le C\|f\|_{\elle1}\int_0^{|\omega|} t^{-\frac{N-p}{N}-\frac{p-1}{N}} \le C |\omega|^{\frac{1}{N}},
\end{equation}
where we used Lemma \ref{dupap}. Estimate \eqref{lowlim} implies that the sequence
\[
\left\{\frac{|\un|^{p-2}\un}{1+\frac1n|\un|^{p-1}}E_n(x)\right\}
\] 
is equi-integrable. This, together with the $a.e.$ convergence of $\un$, allows us to pass to the limit, as $n\to\infty$, also in the lower order term of \eqref{apprp} and conclude that  
\[
\io a(x,\nabla u)\nabla\phi=\io u  |u|^{p-2} \ E(x)\nabla\phi+\io f(x)\phi \ \ \ \forall \ \phi\in C^1_0(\Omega).
\]

Finally from \eqref{aecon} and Proposition \ref{29-10} we easily infer that
\[
|u|^{p-1}\in L^{\frac{N}{N-p},\infi}(\Omega) \ \ \ \mbox{and} \ \ \  |\nabla u|^{p-1}\in L^{\frac{N}{N-1},\infi}(\Omega).
\]
\textbf{Case \emph{(ii)}.} If $p>2-\frac{1}{N}$ and $f\in\lorr{1}{q}$ with $0<q\le\infty$, we infer from Lemma \ref{dupap} that $\{\un\}$ and $\{|\nabla\un|\}$ are bounded in $L^{\frac{(p-1)N}{N-p}, (p-1)q}(\Omega)$ and $L^{\frac{(p-1)N}{N-1},(p-1)q}(\Omega)$ respectively. Since $\frac{N(p-1)}{N-1}>1$ we deduce that there exist $u\in W^{1,r}_0(\Omega)$ such that $\un\rightharpoonup u$ in $W^{1,r}_0(\Omega)$ for any $1<r<\frac{N(p-1)}{N-1}$. Thus following the same arguments of the previous step, we conclude that there exists $u$ distributional solution of \eqref{pdrift} such that
\[
|u|\in \lor{\frac{(p-1)N}{N-p}}{(p-1)q} \ \ \ \mbox{and} \ \ \ |\nabla u|\in\lor{\frac{(p-1)N}{N-1}}{(p-1)q}.
\]
\textbf{Case \emph{(iii)}.} If $p=2-\frac{1}{N}$ and $f\in\lorr{1}{q}$ with $0<q\le\frac{1}{p-1}=\frac{N}{N-1}$, Lemma \ref{dupap} implies that $\{|\nabla \un|\}$ is bounded in $\elle1$. Since $\elle1$ is not reflexive, this is not enough to assure the existence of a weakly converging subsequence. In order to recover a compactness property for $\{|\nabla \un|\}$, we need to prove its equi-integrability. For it, let $\omega$ be a measurable subset of $\Omega$ and notice that
\begin{equation}\label{equint}
\int_{\omega}|\nabla \un(x)|dx\le \int_0^{|\omega|}\overline{|\nabla\un|}(t)dt\le  \int_0^{|\omega|}t\left(\frac{1}{t}\int_0^s\overline{|\nabla \un|}^{p-1}\right)^{\frac{N}{N-1}}\frac{dt}{t}
\end{equation}
\[
\le\int_0^{|\omega|}t\left(\frac{1}{t}\int_0^t\big(v(s)^{p-1}s^{-\frac{p-1}{N}}+\tilde{f}s^{\frac{1}{N}}\big)ds+\left(\frac{1}{t}\int_t^{|\Omega|}\big(v(s)^{p}s^{-\frac pN}+\tilde{f}^{p'}s^{\frac{p'}{N}}\big)ds\right)^{\frac{1}{p'}}\right)^{\frac{N}{N-1}}\frac{dt}{t}
\]
where the last inequality comes from \eqref{10:01}. Lemma \ref{dupap} with $f\in\lorr{1}{\frac{N}{N-1}}$ implies that
\[
 \int_0^{|\Omega|}t\left(\frac{1}{t}\int_0^t\big(v(s)^{p-1}s^{-\frac{p-1}{N}}+\tilde{f}s^{\frac{1}{N}}\big)ds+\left(\frac{1}{t}\int_t^{|\Omega|}\big(v(s)^{p}s^{-\frac pN}+\tilde{f}^{p'}s^{\frac{p'}{N}}\big)ds\right)^{\frac{1}{p'}}\right)^{\frac{N}{N-1}}\frac{dt}{t}
\]
\[
\le C \int_0^{|\Omega|}\left(t\tilde{f}(t)\right)^{\frac{N}{N-1}}\frac{dt}{t}=C\|f\|_{\lorr{1}{\frac{N}{N-1}}}.
\]
This means that the function
\[
\left(\frac{1}{t}\int_0^t\big(v(s)^{p-1}s^{-\frac{p-1}{N}}+\tilde{f}s^{\frac{1}{N}}\big)ds+\left(\frac{1}{t}\int_t^{|\Omega|}\big(v(s)^{p}s^{-\frac pN}+\tilde{f}^{p'}s^{\frac{p'}{N}}\big)ds\right)^{\frac{1}{p'}}\right)^{\frac{N}{N-1}}
\]
belongs to $L^1(0,|\Omega|)$. This consideration and inequality \eqref{equint} imply that for every $\epsilon$ there exists $\delta>0$ such that
\[
\int_{\omega}|\nabla \un(x)|dx\le \epsilon \ \ \ \forall \ \omega\subset\Omega \ \ \ \mbox{with} \ \ \ |\omega|<\delta.
\] 
Hence we take advantage of  Dunford-Pettis Theorem to infer the existence of a vector field $L\in\left(\elle1\right)^N$ such that
\[
\nabla\un\rightharpoonup L \ \ \ \mbox{in} \  \left(\elle1\right)^N.
\]
By the very definition of weak gradient of a Sobolev function it results that
\begin{equation}\label{luth}
\io\nabla\un F=-\io \un \mbox{div}(F) \ \ \ \forall  \ F\in \left(C^{\infi}_0(\Omega)\right)^{N}.
\end{equation}
Thanks to the weak convergence of $\nabla\un$ in $\left(\elle1\right)^N$ and the strong convergence of $\un$ in $\elle1$ (Lemma \ref{dupap} says that indeed $\un$ strongly converge to $u$ in $\elle{r}$ with $1<r<\frac{N}{N-1}$), we can pass to the limit in the equation above and deduce that $F\equiv\nabla u$.\\
At this point, thanks to the almost convergence of $\nabla\un$ to $\nabla u$ (see Lemma \ref{22-3}), we can infer that indeed
\[
\nabla\un \to \nabla u \  \ \ \mbox{in} \ \ \ \left(\elle 1\right)^N.
\]
Since $p-1=1-\frac{1}{N}<1$, we also have that $|\nabla\un|^{p-2}\nabla\un\to |\nabla u|^{p-2}\nabla u$ in $\elle1$. We follow the arguments of the previous step to conclude that $u$ is a solution of \eqref{pdrift}.
Moreover, thanks to the almost everywhere of both $\{\un\}$ and $\{|\nabla \un|\}$, we apply again Proposition \ref{29-10} to conclude that
\[
|u|\in\lor{\frac{N-1}{N-p}}{(p-1)q} \ \ \ \mbox{and} \ \ \ |\nabla u|\in\lor{1}{(p-1)q}.
\]

\textbf{Case (iv).} The case $p<2-\frac{1}{N}$ and $f\in\lor{m}{q}$ with $m=\frac{N}{N(p-1)+1}$ and $0< q\le\frac{1}{p-1}$ is handled similarly to the Case \emph{(iii)}. Indeed, for the considered values of $m$, it results $\frac{(p-1)Nm}{N-pm}=1$, thus Lemma \ref{dupap} implies that $\{|\nabla\un|\}$ is bounded in $\elle1$. Reasoning as in \eqref{equint}, \eqref{luth} and using the almost everywhere convergence of the gradient (see Lemma \ref{22-3}), we conclude that
\[
\nabla\un\to\nabla u \ \ \ \mbox{in}\ \ \ \left(\elle1\right)^N.
\]
From now on the proof is close to the one of the previous case.
\end{proof}
\begin{proof}[Proof of Theorem \ref{teohighp}.]
\textbf{Case \emph{(i)}.} Following the same argument of the first step of the proof of Theorem \ref{teolowp}. We infer that there exists $u\in W^{1,r}_0(\Omega)$ with $1<r<\frac{Nm(p-1)}{N-m}$ such that up to a subsequence
\[
\nabla\un\to \nabla u \ \ \ \mbox{in} \ \ \ L^{\frac{r}{p-1}}(\Omega).
\] 
Since it is possible to chose $r$ such that $\frac{r}{p-1}>1$, we deduce that
\[
|\nabla \un|^{p-2}\nabla \un  \to |\nabla u|^{p-2}\nabla u  \ \ \ \mbox{in} \ \ \ \elle1.
\] 
In order to pass to the limit in \eqref{apprp}, it is enough to notice that \eqref{aecon} and \eqref{lowlim} are still valid. We also have that
\[
|u|^{p-1}\in L^{\frac{N}{N-p},\infi}(\Omega) \ \ \ \mbox{and} \ \ \  |\nabla u|^{p-1}\in L^{\frac{N}{N-1},\infi}(\Omega).
\]
\textbf{Case (ii).} Choosing $\un$ as a test function in \eqref{apprp} and Using H\"older's inequality we get
\[
\alpha\io|\nabla\un|^p\le \left(\io |E|^{p'}|u|^p\right)^\frac{1}{p'}\left(\io|\nabla\un|^p\right)^\frac{1}{p}+\frac{1}{\mathcal{S}}\|f\|_{\elle{(p^*)'}}\left(\io|\nabla\un|^p\right)^\frac{1}{p}\] 
Moreover thanks to \eqref{stop} it results that $\{\un\}$ is bounded in $\elle {q}$ for $p^*<q<[(p-1)m^{*}]^*$. Thus
\[
\int_0^{|\Omega|}t^{-\frac{p}{N}}\overline{v}^p(t)dt\le\left(\int_0^{|\Omega|}\overline{v}^q(t)dt\right)^{\frac{p}{q}}\left(\int_0^{|\Omega|}t^{-\frac{pq}{N(q-p)}}\right)^{\frac{q-p}{q}}\le C
\]
since $1-\frac{pq}{N(q-p)}>0$. Hence
\[
\|\nabla\un\|_{\elle 2}\le\|E\|_{\emme N} \left(\int_0^{|\Omega|}t^{-\frac{p}{N}}\overline{v}^p(t)dt\right)^\frac{1}{p'}+\frac{1}{\mathcal{S}}\|f\|_{\elle{(p^*)'}}\le C
\]
At this point we conclude that up to a subsequence $\{\nabla\un\}$ weakly converge in $W^{1,p}_0(\Omega)$ to a function $u\in W^{1,p}_0(\Omega)$. The rest of the proof is the same of \emph{Case (i)}.
\end{proof}
\subsection{Drift term}

In the next Lemmas we recover the pointwise estimate for the the rearrangement of $\wn$, the solution of \eqref{!!appr1}, and its gradient. 
\begin{lemma}\label{sgo} Let us assume \eqref{ll5} and \eqref{!!marbis}. The sequence $\{\wn\}$ of solution of \eqref{!!appr1} satisfies the following estimates:
\begin{equation}\label{05-10}
\overline{w}_n(\tau)\le z(t):= C\int_{\tau}^{|\Omega|}t^{p'\left(\frac{1}{N}-1\right)+\frac{B}{\alpha\sigma_N(p-1)}}\left(\int_0^{t} \overline{f}(s)s^{-\frac{B}{\alpha\sigma_{N}}}ds\right)^{\frac{1}{p-1}}dt
\end{equation}
and
\begin{multline}\label{05-10bis}
\frac{1}{s}\int_0^s\overline{|\nabla\wn|}^{p-1}\le C_1\left[\frac{1}{s}\int_0^s t^{\frac{1}{N}-1+\frac{B}{\alpha\sigma_{N}}} \left(\int_0^{t} \overline{f}(\tau)\tau^{-\frac{B}{\alpha\sigma_{N}}}d\tau\right)dt\right.\\ \left.+\left(\frac{1}{s}\int_{s}^{|\Omega|}t^{p'\left(\frac{1}{N}-1+\frac{B}{\alpha\sigma_{N}}\right)} \left(\int_0^{t} \overline{f}(\tau)\tau^{-\frac{B}{\alpha\sigma_{N}}}d\tau\right)^{p'}dt\right)^{\frac{1}{p'}}\right],
\end{multline}
where $C$ and $C_1$ are two constant depending on $N,\alpha,p,E,m$.
\end{lemma}
\begin{proof} 
Let us divide the prof in two steps.\\
\textbf{Step 1.} Estimate for $\wn$.\\
\textbf{Step 2.} Estimate for $\nabla \wn$.\\
\textbf{Step 1.} Let us set for any $n\in\mathbb{N}$, $k>0$ and $s\in(0,|\Omega|)$ the distribution function of $\wn$
\[
A_n(k)=|\{|\wn|>k\}|,
\]
and the pseudo rearrangements of the two components of $E(x)$ (see \eqref{!!marbis} and \eqref{night2})
\[
Q_{1,n}(s):=\frac{d}{ds}\int_{\Omega_n(s)}|\mathcal{F}(x)|^pdx \ \ \ \mbox{and} \ \ \ Q_{2,n}(s):=\frac{d}{ds}\int_{\Omega_n(s)}|\mathcal{E}(x)|^pdx.
\]
As in Lemma \ref{regup}, let us take $\frac{T_h(\Gk(\un))}{h}$ with $h>0$ and $k\ge0$ as test function in \eqref{!!appr1}. We obtain that
\begin{equation}\label{samsara}
\frac{\alpha}{h}\int_{\{k<|\wn|<k+h\}}|\nabla\wn|^p\le \int_{\{|w_n|>k\}}|f|+\int_{\{|w_n|>k\}}|E||\nabla \wn|^{p-1}.
\end{equation}
Recalling \eqref{24-5}, let us note that the last integral above can be estimate as
\[
\ioo{k}|E_n(x)||\nabla\wn|^{p-1}=\int_k^{+\infty}\left(\frac{d}{ds}\ioo{s}|E_n(x)||\nabla\wn|^{p-1}\right)ds
\]
\[
\le\int_k^{+\infty}\big(Q_{1,n}(A_n(s))^{\frac{1}{p}}+Q_{2,n}(A_n(s))^{\frac{1}{p}}\big)(-A_n'(s))^{\frac1p}\left(-\frac{d}{ds}\ioo{s} |\nabla \wn|^p\right)^{\frac{1}{p'}}ds,
\]
Passing to the limit as $h\to 0$ in \eqref{samsara}, we recover that
\[
-\frac{d}{dk}\ioo{k}|\nabla \wn|^p\le\frac{1}{\alpha}\ioo{k}|f|\]\[+ \frac{1}{\alpha}\int_k^{+\infty}\mathcal{Q}_n(A_n(s))(-A'_n(s))^{\frac1p}\left(-\frac{d}{ds}\ioo{s} |\nabla \wn|^p\right)^{\frac{1}{p'}}ds,
\]
where we have set $\mathcal{Q}_n(s)=Q_{1,n}(s)^{\frac{1}{p}}+Q_{2,n}(s)^{\frac{1}{p}}$ to have a more compact notation.
Using \eqref{talenti} we obtain
\[
\left(-\frac{d}{dk}\ioo{k}|\nabla \wn|^p\right)^{\frac{1}{p'}}\le\frac{1}{\alpha\sigma_N}A_n(k)^{\frac{1}{N}-1}(-A'_n(k))^{\frac{1}{p'}}\ioo{k}|f|\]
\[+\frac{1}{\alpha\sigma_N}A_n(k)^{\frac{1}{N}-1}(-A'_n(k))^{\frac{1}{p'}} \int_k^{+\infty}\mathcal{Q}_n(A_n(s))(-A'_n(s))^{\frac{1}{p}}\left(-\frac{d}{ds}\ioo{s} |\nabla \wn|^p\right)^{\frac{1}{p'}}ds.
\] 

Let us use Lemma \ref{gron} and make a change of variable to obtain that
\[
\left(-\frac{d}{dk}\ioo{k}|\nabla \wn|^p\right)^{\frac{1}{p'}}\le\frac{1}{\sigma_N}A_n(k)^{\frac{1}{N}-1}(-A'_n(k))^{\frac{1}{p'}}\int_0^{A_n(k)}\overline{f}\]
\[+\frac{1}{\alpha^2\sigma_N^2}A_n(k)^{\frac{1}{N}-1}(-A'_n(k))^{\frac{1}{p'}} \int_0^{A_n(k)}\mathcal{Q}_n(s)s^{\frac{1}{N}-1}\left(\int_0^s\overline{f}(\tau)d\tau\right)e^{\frac{1}{\alpha\sigma_N}\int_s^{A_n(k)}\mathcal{Q}_n(\tau)\tau^{\frac{1}{N}-1}d\tau}ds.
\]
We note that the integral in the second line above can be written as
\[
-\alpha\sigma_N\int_0^{A_n(k)}\left(\int_0^s\overline{f}(\tau)d\tau\right)\frac{d}{ds}\left(e^{\frac{1}{\alpha\sigma_N}\int_s^{A_n(k)}\mathcal{Q}_n(\tau)\tau^{\frac{1}{N}-1}d\tau}\right)ds.
\]
Thus integrating by parts we finally obtain
\begin{multline}\label{nirv}
\left(-\frac{d}{dk}\ioo{k}|\nabla\wn|^p\right)^{\frac{1}{p'}}\le\\ \frac{1}{\alpha\sigma_N}A_n(k)^{\frac{1}{N}-1}(-A'_n(k))^{\frac{1}{p'}} \int_0^{A_n(k)} \overline{f}(s)e^{\frac{1}{\alpha\sigma_N}\int_s^{A_n(k)}\left(Q_{1,n}(\tau)^{\frac{1}{p}}+Q_{2,n}(\tau)^{\frac{1}{p}}\right)\tau^{\frac{1}{N}-1}d\tau}ds.
\end{multline}
Using once more \eqref{talenti}, estimate \eqref{nirv} becomes
\[
1\le\frac{1}{\alpha\sigma_N^{p}}A_n(k)^{p\left(\frac{1}{N}-1\right)}(-A'_n(k))^{p-1} \int_0^{A_n(k)} \overline{f}(s)e^{\frac{1}{\alpha\sigma_N}\int_s^{A_n(k)}\left(Q_{1,n}(\tau)^{\frac{1}{p}}+Q_{2,n}(\tau)^{\frac{1}{p}}\right)\tau^{\frac{1}{N}-1}d\tau}ds,
\]
and by a change of variable
\begin{equation}\label{talassa}
-\frac{d}{dt}\overline{w}_n(t)\le \frac{1}{\alpha^{\frac{1}{p-1}}\sigma_N^{p'}}t^{p'\left(\frac{1}{N}-1\right)}\left(\int_0^{t} \overline{f}(s)e^{\frac{1}{\alpha\sigma_N}\int_s^{t}\left(Q_{1,n}(\tau)^{\frac{1}{p}}+Q_{2,n}(\tau)^{\frac{1}{p}}\right)\tau^{\frac{1}{N}-1}d\tau}ds\right)^{\frac{1}{p-1}}.
\end{equation}
By construction and by Lemma \ref{inl} we deduce that $\|Q_{1,n}\|_{\elle{\infty}}\le C \|\mathcal{F}\|_{\elle{\infty}}^p$ and moreover, by means of Young Inequality and integration by parts, we have that
\[
\int_s^{t}Q_{2,n}(\tau)^{\frac1p}\tau^{\frac{1}{N}-1}d\tau \le  \frac{NB}{p(N-p)}+B\log\left(\frac{t}{s}\right).
\]
Thus integrating \eqref{talassa} we recover \eqref{05-10}.\\
\textbf{Step 2.} Recalling Lemma \ref{16-8} and Remark \ref{okkio}, we obtain that
\[
\int_0^s\overline{|\nabla  \wn|}^{p-1}d\tau =\int_{\widetilde{\Omega}_n(s)}|\nabla  \wn|^{p-1}dx\]
\[= \int_{\widetilde{\Omega}_n(s)\cap\{| \wn|>\overline{w}_n(s)\}}|\nabla \wn|^{p-1}dx+\int_{\widetilde{\Omega}_n(s)\cap\{| \wn|\le\overline{w}_n(s)\}}|\nabla  \wn|^{p-1}dx\]
\[
\le\int_{\{| \wn|>\overline{w}_n(s)\}}|\nabla  \wn|^{p-1}dx+
 \left(\int_{\{| \wn|\le\overline{w}_n(s)\}}|\nabla  \wn|^{p}dx\right)^{\frac{1}{p'}}|\widetilde{\Omega}_n(s)|^{\frac{1}{p}}\le I_1(s)+I_2^{\frac{1}{p'}}(s)s^{\frac{1}{p}}.
\]
\textbf{Estimate of $I_2$.} From \eqref{nirv} we also have
\[
-\frac{d}{dk}\ioo{k}|\nabla \wn|^p\le C A_n(k)^{p'\left(\frac{1}{N}-1+\frac{B}{\alpha\sigma_N}\right)} \left(\int_0^{A_n(k)}t^{-\frac{B}{\alpha\sigma_N}} \overline{f}(t)dt\right)^{p'}(-A'_n(k)),
\]
from which we infer that (see \eqref{composition})
\[
\frac{d}{ds}\int_{\{| \wn|>\overline{w}_n(s)\}}|\nabla  \wn|^{p}\le C s^{p'\left(\frac{1}{N}-1+\frac{B}{\alpha\sigma_N}\right)}\left(\int_0^{s}t^{-\frac{B}{\alpha\sigma_N}} \overline{f}(t)dt\right)^{p'}.
\]
Integrating between $\tau$ and $\Omega$ we get
\[
I_2\le C \int_\tau^{|\Omega|}s^{p'\left(\frac{1}{N}-1+\frac{B}{\alpha\sigma_N}\right)}\left(\int_0^{s}t^{-\frac{B}{\alpha\sigma_N}} \overline{f}(t)dt\right)^{p'}ds.
\]
\textbf{Estimate of $I_1(s)$.} 
As far as $I_1$ is concerned, recalling \eqref{06/10}, it follows  
\[
\frac{d}{ds}\int_{\{|\wn|>\overline{u}_n(s)\}}|\nabla \wn|^{p-1}\le \left(\frac{d}{ds}\int_{\{|\wn|>\overline{u}_n(s)\}}|\nabla \wn|^p\right)^{\frac{1}{p'}}\le  C s^{\frac{1}{N}-1+\frac{B}{\alpha\sigma_N}}\int_0^{s}t^{-\frac{B}{\alpha\sigma_N}} \overline{f}(t)dt.
\]

Integrating between $0$ and $\tau$ we get
\[
I_1=\ioo{\overline{w}_n(\tau)}|\nabla\wn|^{p-1}\le C \int_0^{\tau}s^{\frac{1}{N}-1+\gamma}\int_0^{s}t^{-\gamma} \overline{f}(t)dtds.
\]
Putting together these two pieces of information we obtain \eqref{05-10bis}.
\end{proof}
Let us provide now the a priori bound for $\{\wn\}$ and $\{|\nabla\wn|\}$ in the required Lorentz spaces.
\begin{lemma}\label{!!dupa1}
There exist two constant $C=C(\alpha,p,E,N)$ and $\tilde{C}=\tilde{C}(\alpha,p,E,N)$ such that
\[
\mbox{if} \ \ \ 1<m<\frac{N}{p}, \ \ \ 0<q\le\infty \ \ \ \mbox{then} \ \ \ \|\wn\|_{\lor{[(p-1)m^{*}]^*}{(p-1)q}}\le C \|f\|_{\lor{m}{q}}\]
and
\[ if \ \ \  1<m<\left(p^*\right)',\ \ \ 0<q\le\infty   \ \ \ \mbox{then} \ \ \  \|\nabla \wn\|_{\lor{(p-1)m^{*}}{(p-1)q}}\le\tilde{C}\|f\|_{\lor{m}{q}}.
\]
\end{lemma}
\begin{proof}
\emph{Estimate for $\{\wn\}$.} Assume tha $q>\infty$. From \eqref{05-10} it follows that
\[
\|\wn\|^q_{L^{\frac{(p-1)Nm}{N-pm},(p-1)q}(\Omega)}=\int_0^{+\infi}t^{\frac{q(N-pm)}{Nm}}\overline{w}_n(t)^{q}\frac{dt}{t}\]
\[\le C\int_0^{+\infi}\tau^{\frac{q(N-pm)}{Nm}+(p-1)q}\left(\frac{1}{\tau}\int_{\tau}^{|\Omega|}t^{p'\left(\frac{1}{N}-1\right)+\frac{B}{\alpha\sigma_N(p-1)}}\left(\int_0^{t} \overline{f}(s)s^{-\frac{B}{\alpha\sigma_{N}}}ds\right)^{\frac{1}{p-1}}dt\right)^{(p-1)q}\frac{d\tau}{\tau}\]
\[
\le C\left[\int_0^{+\infty}\tau^{\frac{q}{m}+\frac{qB}{\alpha\sigma_N}}\left(\tau^{-1}\int_0^{\tau}t^{-\frac{B}{\alpha\sigma_N}} \overline{f}(t)dt\right)^q\frac{d\tau}{\tau}\right]
 \le  C\int_0^{\infi}t^{\frac{q}{m}}\overline{f}^q\frac{dt}{t} ,
\]
where we used Lemma \ref{trickint} twice, once with $\delta=\frac{N-pm}{(p-1)Nm}+1>1$ and the second time with $\delta=\frac{1}{m}+\frac{B}{\alpha\sigma_N}<1$.
If $q=\infty$ directly from \eqref{05-10} we obtain that
\[
\overline{w}_n\le C \|f\|_{\lor{m}{\infty}}t^{-\frac{N-pm}{(p-1)Nm}}.
\]

\emph{Estimate for $\{|\nabla\wn|\}$.} Let us start with 
\[
\int_0^{\infty}\tau^{\frac{q}{m^*}}\left(\frac{1}{\tau}\int_0^{\tau}s^{\frac{1}{N}-1+\frac{B}{\alpha\sigma_N}}\int_0^{s}t^{-\frac{B}{\alpha\sigma_N}} \overline{f}(t)dtds\right)^q\frac{d\tau}{\tau} 
\]
\[
\le C\int_0^{\infty}\tau^{\frac{q}{m}+q\frac{B}{\alpha\sigma_N}}\left(\frac{1}{\tau}\int_0^{\tau}t^{-\frac{B}{\alpha\sigma_N}}\overline{f}(t)dt\right)^q\frac{d\tau}{\tau}
\]
\[
\le C\int_0^{\infty}\tau^{\frac{q}{m}}\overline{f}(\tau)^q\frac{d\tau}{\tau}.
\]
Moreover
\[
\int_0^{\infty}\tau^{\frac{q}{m^*}}\left(\frac{1}{\tau}\int_{\tau}^{|\Omega|}t^{p'\left(\frac{1}{N}-1+\frac{B}{\alpha\sigma_{N}}\right)} \left(\int_0^{t} \overline{f}(\tau)\tau^{-\frac{B}{\alpha\sigma_{N}}}d\tau\right)^{p'}dt\right)^{\frac{q}{p'}}\frac{d\tau}{\tau}
\]
\[
\le C \int_0^{\infty}\tau^{\frac{q}{m}+\frac{qB}{\alpha\sigma_N}}\left(\frac{1}{s}\int_0^{\tau}t^{-\frac{B}{\alpha\sigma_N}}\overline{f}(t)dt\right)^q\frac{d\tau}{\tau}\le C\int_0^{\infty}\tau^{\frac{q}{m}}\overline{f}^q\frac{d\tau}{\tau},
\]
where we used Lemma \ref{trickint} twice, once with $\delta=\frac{p'}{m^*}>1$ and the second time with $\delta=\frac{1}{m}+\frac{B}{\alpha\sigma_N}<1$.
Hence we have that
\[
\|\nabla\wn\|_{\lor{\frac{(p-1)Nm}{N-m}}{(p-1)q}}^q\le\int_0^{\infty}\tau^{\frac{q}{m^*}}\left(\frac{1}{s}\int_0^\tau\overline{|\nabla\wn|}^{p-1}(t)dt\right)^q\frac{d\tau}{\tau}\le\int_0^{\infty}\tau^{\frac{q}{m}}\overline{f}^q\frac{d\tau}{\tau}.
\]
If $q=\infty$ directly from \eqref{05-10bis} we obtain that
\[
\overline{|\nabla\wn|}\le C \|f\|_{\lor{m}{\infty}}t^{-\frac{N-m}{(p-1)Nm}}.
\]
\end{proof}

\begin{proof}[Proof of Theorem \ref{!!teohighp}.] \textbf{Case (i).} From Lemma \ref{!!dupa1} we infer the existence of a function $w\in W^{1,r}_0(\Omega)$ with $1<r<\frac{(p-1)Nm}{N-m}$ such that, up to a subsequence
\[
\wn\rightharpoonup w \ \ \ \mbox{in} \ \ \ W^{1,r}_0(\Omega).
\]
This weak converge and the almost everywhere convergence of $\nabla\wn$ proved in \eqref{aebis} allow us to conclude (see \eqref{9/2/2019}) that
\[
|\nabla \wn|^{p-2}\nabla \wn  \to |\nabla w|^{p-2}\nabla w  \ \ \ \mbox{in} \ \ \ \elle1.
\] 
In order to deal with the lower order term notice that for any subset $A\subset\Omega$ it results (recall that $m>1$)
\[
\int_{A}|\nabla\wn|^{p-1}|E_n(x)|\le \int_0^{|A|}\overline{|\nabla\wn|}(s)\overline{E}(s)ds\le C\int_0^{|A|}t^{-\frac{1}{m^*}-\frac{1}{N}}\le C |A|^{\frac{1}{m'}},
\]
that is the equi-integrability of the sequence 
\[
\left\{\frac{\nabla\wn\cdot E(x)}{1+\frac{1}{n}|\nabla\wn|}\right\}.
\]
This and the almost everywhere convergence of the gradients assured by Lemma \ref{aenodiv} allows us to conclude that the function $w$ satisfies \eqref{!!pdriftw}. Moreover thanks to Proposition \ref{29-10} it follows that
\[
\|w\|_{\lor{\frac{(p-1)Nm}{N-pm}}{(p-1)q}}+\|\nabla w\|_{\lor{\frac{(p-1)Nm}{N-m}}{(p-1)q}}\le C\|f\|_{\emme m}.
\]

\textbf{Case (ii).} Form Lemma \ref{!!dupa1} we know that $\{\wn\}$ is bounded in $\elle {q}$ for $p^*<q<[(p-1)m^*]^*$. Thus
\[
\io |E|^p|w|^p\le\|E\|_{\emme N}\int_0^{|\Omega|}t^{-\frac{N}{p}}\overline{w}_n^p(t)dt\le\]\[\|E\|_{\emme N}\left(\int_0^{|\Omega|}\overline{w}_n^q(t)dt\right)^{\frac{2}{q}}\left(\int_0^{|\Omega|}t^{-\frac{pq}{N(q-p)}}\right)^{\frac{q-p}{q}}\le C
\]
since $1-\frac{pq}{N(q-p)}>0$. Let us take now $\wn$ as a test function in \eqref{!!appr1}. Using H\"older's inequality we get
\[
\alpha\io|\nabla\wn|^p\le \left(\io |E|^p|w|^p\right)^\frac{1}{p}\left(\io|\nabla\wn|^p\right)^{\frac1{p'}}+\frac{1}{\mathcal{S}}\|f\|_{\elle{(p^*)'}}\left(\io|\nabla\wn|^p\right)^{\frac1{p}}\]\[\le C\left(\io|\nabla\wn|^p\right)^{\frac1{p'}}+\frac{1}{\mathcal{S}}\|f\|_{\elle{(p^*)'}}\left(\io|\nabla\wn|^p\right)^{\frac1{p}}.
\]
Hence up to a subsequence $\{\nabla\wn\}$ weakly converge in $W^{1,p}_0(\Omega)$ to a function $w\in W^{1,p}_0(\Omega)$. The rest of the proof is the same of \emph{Case (i)}.
\end{proof}

\end{document}